\documentclass[11pt,a4paper]{amsart}
\usepackage[utf8]{inputenc}
\usepackage[english]{babel}
\usepackage{amsmath}
\usepackage[psamsfonts]{amssymb}
\usepackage{mathrsfs}
\usepackage{amsthm}

\usepackage{verbatim,euscript,eufrak}
\usepackage{caption}
\usepackage{graphicx}

\usepackage{color}
\DeclareMathOperator\cardd{card}

\newcommand{\ZZ}{\mathbb Z}

\newcommand{\TT}{\mathbb T}

\newcommand{\pat}{\partial_t}
\newcommand{\pax}{\partial_x}

\newcommand{\pa}{\partial}
\newcommand{\na}{\nabla}

\newcommand{\etl}{e^{\sigma(t) \Lambda}}

\newcommand{\vertiii}[1]{{\left\vert\kern-0.25ex\left\vert\kern-0.25ex\left\vert #1 
		\right\vert\kern-0.25ex\right\vert\kern-0.25ex\right\vert}}

\newcommand{\norm}[1]{\lVert#1\rVert}
\newcommand{\normm}[1]{\bigl\lVert#1\bigr\rVert}
\newcommand{\ep}{\epsilon}
\usepackage[pdfauthor={...},pdftitle={...},bookmarks,colorlinks]{hyperref}

\newcounter{comentcount}
\newcounter{teocount}
\newtheorem{lem}{Lemma}

\newtheorem{corol}{Corollary}
\newtheorem{teo}[teocount]{Theorem}  
\newtheorem{defi}{Definition}

\newtheorem{remark}{Remark}

\title[]{On the dynamics of 3D electrified falling films}

\author[J. He]{Jiao He}
\author[R. Granero-Belinch\'{o}n]{Rafael Granero-Belinch\'{o}n}
\begin{document}
\begin{abstract}
	In this article, we consider a non-local variant of the Kuramoto-Sivashinsky equation in three dimensions (2D interface). Besides showing the global wellposedness of this equation we also obtain some qualitative properties of the solutions. In particular, we prove that the solutions become analytic in the spatial variable for positive time, the existence of a compact global attractor and an upper bound on the number of spatial oscillations of the solutions. We observe that such a bound is particularly interesting due to the chaotic behavior of the solutions.
\end{abstract}

\maketitle

\section{Introduction}
The present work is concerned with the full 3D dynamics of a thin fluid film falling along with a flat inclined plate. Besides gravitational effects, we consider the action of an electric field acting normal to the plate. In particular, for the case where the fluid lies on top of the plate (overlying films), the following equation was derived by Tomlin, Papageorgiou \& Pavliotis \cite{tomlin2017three}:
\begin{align}\label{KStwo}
\eta_{t} + \eta \eta_{x}+ (\beta-1) \eta_{xx}-\eta_{yy}- \gamma \Lambda^{3} \eta+\Delta^{2} \eta = 0
\end{align}
where  $\beta > 0$ is the Reynolds number, $0 \leq \gamma \leq 2$ measures the electric field strength and $\Lambda$ is a non-local operator corresponding to the electric field effect given on the Fourier variables as
$$
\widehat{\Lambda u }= \vert \mathbb{\xi} \vert \hat{u}(\xi) =  (\xi_{1}^{2}+\xi_{2}^{2})^{0.5}\hat{u}(\xi).$$ 
We observe that the term corresponding to the electric field, $- \gamma \Lambda^{3} (\eta),$ always has a destabilizing effect, while the term $(\beta-1) \eta_{xx}$ can be stabilizing or destabilizing depending on the value of the Reynolds number. Namely, for subcritical Reynolds numbers $0<\beta<1$, $(\beta-1) \eta_{xx}$ is a stabilizing term, while for supercritical Reynolds numbers $1<\beta$, it has a destabilizing effect.

Since falling films have received much attention from many authors, a wide variety of results about their nonlinear stability can be found. In particular, the 2D case (1D interface) was first studied by Gonz\'alez \& Castellanos \cite{gonzalez1996nonlinear}. These authors identified a critical electric field strength for sub-critical Reynolds number flows above which instability was found. Later on, Tseluiko $\&$ Papageorgiou also considered the 2D case (1D interface). In particular, Tseluiko \& Papageorgiou performed a numerical study of the 1D analog of \eqref{KStwo} and found attractors for the dynamics for certain values of the physical parameters \cite{tseluiko2006global}. The same authors provided analytical bounds on the energy of the solutions and the dimension of the attractor \cite{tseluiko2006wave} (see also \cite{tseluiko2010dynamics} for the case of vertical film flow). Compared with the case of 1D interface, equation \eqref{KStwo} generalizes previous works by taking transverse dynamics into consideration.

Equation \eqref{KStwo} is similar to the well-known Kuramoto-Sivashinsky equation in one dimension, 
\begin{align}\tag{KS}\label{KS}
\eta_{t} + \eta \eta_{x}= - \eta_{xx} - \eta_{xxxx}
\end{align}
which is a model appearing in several applications. For instance, LaQuey, Mahajan, Rutherford $\&$ Tang \cite{laquey1975nonlinear} obtained \eqref{KS} as a model of collisional trapped-ion mode in tokamak geometry (see also Cohen, Krommes, Tang $\&$ Rosenbluth \cite{cohen1976non}), Kuramoto $\&$ Tsuzuki \cite{kuramoto1976persistent} considered the possible instabilities of a two components reaction-diffusion system and also recovered \eqref{KS}. Furthermore, Sivashinsky \cite{sivashinsky1977nonlinear} (see also the companion paper by Michelson \& Sivashinsky \cite{michelson1977nonlinear}) derived \eqref{KS} as a model of the evolution of a disturbed plane flame front. Later on, Sivashinsky $\&$ Michelson \cite{sivashinsky1980irregular} linked \eqref{KS} to the evolution of a film of viscous liquid flowing down a vertical plane. Several equations sharing some similarities where obtained by Topper \& Kawahara \cite{topper1978approximate}, Lee and Chen \cite{lee1982nonlinear}, Coward \& Hall \cite{coward1993nonlinear}, Frenkel \& Indireshkumar \cite{frenkel1999wavy} and by James $\&$ Wilczek \cite{james2017vortex} when considering falling fluid films, plasma turbulence and cellular suspensions.

Equation \eqref{KS} has rich dynamics. Indeed, applying the Fourier transformation to the linear part of \eqref{KS},
$$\pat \hat{\eta}(\xi) = (\xi^{2}-\xi^{4}) \hat{\eta}(\xi), $$
it results in the stability of high frequencies ($|\xi| > 1$) and instability of low frequencies ($0 < |\xi|< 1$). Specifically, the term $\eta_{xx}$ leads to instability at large scales; the dissipative term $\eta_{xxxx}$ is responsible for damping at small scales. Then we see that for general initial data, the linear problem is unstable and leads to an exponential growth of certain frequencies. When the nonlinear term $\eta \eta_{x}$ is added, stabilization occurs as this term transfers energy from the long wavelengths to the short wavelengths and balances the exponential growth due to the linear parts. This interaction between the unstable linear parts and a nonlinearity who carries energy between frequencies makes the solution of \eqref{KS} to develop chaotic dynamics for certain values of the parameters.

This nonlinear stabilization of the Kuramoto-Sivashinsky equation with $L$-periodic boundary conditions,
$$\eta (x+ L, t)= \eta(x, t), \; \text{for all}\; x \;\text{and} \; t,$$
was considered mathematically by Nicolaenko, Scheurer $\&$ Temam in \cite{nicolaenko1985some} under the hypothesis that the initial data has odd symmetry: $\eta_0(x)=-\eta_0(-x)$. After that, Ilyashenko \cite{il1992global}, Collet, Eckmann, Epstein $\&$ Stubbe \cite{collet1993global} and Goodman  \cite{goodman1994stability} found new bounds for the $L^{2}$-norm of the solution of the KS equation without oddness condition for the initial data. The fact that the solutions are uniformly bounded leads us to the question of the optimal bound for the radius of the absorbing set in $L^2$ for arbitrarily large periods $L$. In that regards, the known bounds are
$$
\limsup_{t\rightarrow\infty}\left(\int_{0}^L u^2dx\right)^{0.5}\leq O(L^p)
$$
where the original $p=5/2$ \cite{nicolaenko1985some} was later improved to $p=8/5$ \cite{collet1993global}
and finally to $p=3/2$ \cite{bronski2006uncertainty}. The global bound has been upgraded recently by Giacomelli $\&$ Otto \cite{giacomelli2005new}, where they proved the bound
$$
\limsup_{t\rightarrow\infty}\left(\int_{0}^L u^2dx\right)^{0.5}\leq o(L^{1.5}).
$$
We observe that the conjectured bound is $O(L^{0.5})$.

The analyticity of solutions is of great interest not only for KS equation, but also for other nonlinear partial differential equations. For instance, we refer the reader to the seminar paper by Foias $\&$ Temam  \cite{foias1989gevrey} where they show that solutions of the Navier-Stokes equations are analytic in time with values in a Gevrey class of functions (in space). This technique has been extended largely to other nonlinear parabolic equations and, in particular, Collet, Eckmann, Epstein $\&$ Stubbe \cite{collet1993analyticity} addressed the spatial analyticity of solutions of one-dimensional Kuramoto-Sivashinsky equation. They showed that at large time the solutions are analytic in a strip around the real axis and also gave a rigorous lower bound for its width, i.e. the radius of analyticity is proportional to $L^{-16/25}$. Gruji{\'c} \cite{grujic2000spatial} used a Gevrey class technique to obtain a neighborhood in the global attractor of the set of all stationary solutions in which the radius of analyticity is independent of $L$. This latter result shed some light on a conjecture in \cite{collet1993analyticity} that asks whether there is a $\alpha>0$, independent of $L$, such that the solutions of the KS equation are analytic in space in the complex strip $\{x+is, s<\alpha\}$ for sufficiently large time. In higher dimensions, the literature on estimating the radius of analyticity for the Kuramoto-Sivashinsky-type equations is more scarce. For example, we refer to the works by Pinto \cite{pinto1999nonlinear, pinto2001analyticity} where, among other properties, the author studied the time analyticity of a variant of the two-dimensional KS equation. More recently, Ioakim $\&$ Smyrlis \cite{ioakim2014analyticity} also studied the analyticity properties of solutions of Kuramoto-Sivashinsky type equations and some related systems.

The goal of the present work is to mathematically study the initial value problem for nonlocal two-dimensional Kuramoto-Sivashinsky-type equation with periodic boundary conditions and initial data with zero mean
$$\int_{0}^{L} \int_{0}^{L} \eta_{0}(x,y) dx dy=0. $$
Of course, the zero average condition is propagated by the PDE. We organize this paper as follows. In section 2, we give some notations, definitions and classical results. In section 3, we show the global existence of solutions to initial value problem \eqref{KStwo} and in section 4, we prove the existence of an absorbing set in $L^{2}$ and in higher Sobolev norms. In section 5, we prove that these solutions are analytic in a strip based on \textit{a priori} estimates in a Gevrey class. Finally, in section 6, we establish a bound for the number of spatial oscillations which are a manifestation of the spatial chaos that this PDE evidences.
\subsection{Notation}
We will use $C$ to denote a universal constant that can change from one line to another.  We will make frequent use of the usual homogeneous Sobolev spaces $\dot{H}^{s}$:
$$ \dot{H}^{s}(\TT^{2}) = \big\{\eta \in L^{2}(\TT^{2}): \sum_{\xi \in \ZZ^{2}}  |\xi|^{2s} |\widehat{\eta}(\xi)|^{2} < \infty \big\}$$
where $\widehat{\eta}  (\xi)$ is the Fourier series of $\eta$ : 
\begin{equation*}
\widehat{\eta}(\xi) \stackrel{\text { def }}{=}(2 \pi)^{-2} \int_{\mathbb{T}^{2}} e^{-i \xi_1  x - i \xi_2 y} \eta(x, y) d x d y,\quad \xi = (\xi_1, \xi_2) \in \mathbb{Z}^{2}.
\end{equation*}

\section{Rescaling of the equation}
We assume that $\eta$ is $L$-periodic, 
$$
\eta(x+L,y)=\eta(x,y),\;\;
\eta(x,y+L)=\eta(x,y),
$$
and define $\TT^{2} = [0,2\pi]^{2}$ and $\lambda=\frac{2 \pi}{L}$. We rescale our variables according to
$$\tilde{x} = \lambda  x,\qquad  \tilde{y} = \lambda y , \qquad  \tilde{\eta} =\lambda^{-1}  \eta,\qquad \tilde{t}= \lambda^{2} t, $$
which gives 
\begin{align*}
\lambda^{3} \tilde{\eta_{t}} + \lambda^{3} \tilde{\eta} \tilde{\eta}_{x}+ (\beta-1) \lambda^{3} \tilde{\eta}_{xx}- \lambda^{3}\tilde{\eta}_{yy}- \gamma \lambda^{4} \Lambda^{3} (\tilde{\eta})+ \lambda^{5} \Delta^{2} \tilde{\eta} = 0.
\end{align*}
Then we obtain 
\begin{align}
 \tilde{\eta_{t}} + \tilde{\eta} \tilde{\eta}_{x}+ (\beta-1)  \tilde{\eta}_{xx}-\tilde{\eta}_{yy}- \gamma \lambda \Lambda^{3}(\tilde{\eta})+ \lambda^{2} \Delta^{2} \tilde{\eta} = 0.
\end{align}
Denoting $\delta = \gamma \lambda$ and $ \ep= \lambda^{2}$, we can equivalently consider the following initial-value problem
        \begin{equation}\label{KStwo-dimension}
\eta_{t}+ \eta \eta_{x}+ (\beta-1)  \eta_{xx}-\eta_{yy}- \delta \Lambda^{3}(\eta)+ \ep \Delta^{2} \eta = 0, \; (x,y)\in \TT^{2}, t >0
    \end{equation} 
    with initial data 
    $$\eta(x,y,0)  = \eta_{0}(x,y), \;  (x,y)\in \TT^{2}.$$
   
In what follows, we will drop the tilde notation.

\section{Global existence of strong solutions}
In this section, we will state the global well-posedness result of the initial-value problem \eqref{KStwo-dimension}:
\begin{teo}\label{globalexistence}
If $\eta_{0} \in H^{2}(\TT^{2})$, then for every $0<T< \infty $ the initial value problem \eqref{KStwo-dimension} has a unique solution
$$\eta \in C([0,T]; H^{2}(\TT^{2}))\cap L^2(0,T;H^4(\TT^{2})).$$ 
\end{teo}
\begin{proof}[Proof]
\textbf{Step 1 : $L^{2}$ estimate.} We multiply \eqref{KStwo-dimension} by $\eta$ and integrate by parts to obtain
\begin{align*}
\frac{1}{2}\frac{d}{dt} \norm{\eta}^{2}_{L^{2}}
&= (\beta-1) \norm{\eta_{x}}^{2}_{L^{2}} - \norm{\eta_{y}}_{L^{2}}^{2} + \delta \norm{\Lambda^{\frac{3}{2}} \eta}_{L^2}^2 - \ep \norm{ \Delta\eta}^{2}_{L^{2}}\\
& \leq C(\beta,\delta) \norm{\eta}_{H^{\frac{3}{2}}}^2 - \ep \norm{\eta}^{2}_{H^{2}}\\
& \leq C(\beta,\delta) \norm{\eta}_{H^2}^{3/2} \norm{\eta}_{L^2}^{1/2} - \ep \norm{\eta}^{2}_{H^{2}}
\end{align*}
By Young's inequality, we find that 
$$
\frac{d}{dt} \norm{\eta}^{2}_{L^{2}}
\leq -\ep \norm{\Delta \eta}^{2}_{L^{2}} + C(\ep, \beta,\delta)
 \norm{\eta}^{2}_{L^{2}}
$$
where $C(\ep, \beta,\delta)$ is a constant depending on $\ep, \beta, \delta$ and may change line by line. 
An application of Gronwall's inequality leads us to
\begin{equation*}
\norm{\eta}_{L^{2}}^{2} + \ep \int_{0}^{t} \exp \left(  C(\ep, \beta,\delta)\left(t-s\right) \right) \norm{ \Delta \eta}^{2}_{L^{2}} ds 
\leq 
  C(\ep, \beta,\delta) \norm{\eta_{0}}_{L^{2}}^{2} e^t.
\end{equation*}
Hence,
\begin{equation*}
\norm{\eta}_{L^{2}}^{2} + \ep \int_{0}^{t} \norm{\eta}^{2}_{H^{2}} ds 
\leq 
  C(\ep, \beta,\delta) \norm{\eta_{0}}_{L^{2}}^{2} e^t.
\end{equation*}

\textbf{Step 2 : $H^{1}$ estimate.} Now we multiply \eqref{KStwo-dimension} by $- \Delta \eta$ and integrate by parts to obtain that
\begin{align*}
\frac{1}{2}\frac{d}{dt} \norm{\eta}^{2}_{H^{1}} 
&= (\beta-1) \norm{\eta_{x}}^{2}_{H^{1}} - \norm{\eta_{y}}_{H^{1}}^{2} + \delta \norm{\Lambda^{\frac{3}{2}}\eta}^{2}_{H^{1}} - \ep \norm{\Delta \eta}^{2}_{H^{1}}+ \norm{\eta}_{L^{\infty}} \norm{\eta}_{H^{1}} \norm{\Delta \eta}_{L^{2}}.
\end{align*}
Using the same method as in step 1, we get
\begin{align*}
\frac{1}{2}\frac{d}{dt} \norm{\eta}^{2}_{H^{1}} 
& \leq -\frac{3\ep}{4} \norm{\Delta \eta}^{2}_{H^{1}} +   C(\ep, \beta,\delta)\norm{\eta}^{2}_{H^{1}} + \frac{1}{2\ep}  \norm{\eta}^{2}_{L^{\infty}} \norm{\eta}^{2}_{H^{1}} + \frac{\ep}{2} \norm{\Delta \eta}^{2}_{L^{2}}\\
& \leq -\frac{\ep}{4} \norm{\Delta \eta}^{2}_{H^{1}}+
\left( C(\ep, \beta,\delta) + \frac{1}{2\ep}  \norm{\eta}^{2}_{L^{\infty}} \right) \norm{\eta}^{2}_{H^{1}},
\end{align*}
which implies
\begin{align*}
\frac{d}{dt} \norm{\eta}^{2}_{H^{1}} 
& \leq -\frac{\ep}{2} \norm{\Delta \eta}^{2}_{H^{1}}+ \left(  C(\ep, \beta,\delta) +\frac{1}{\ep}  \norm{\eta}^{2}_{L^{\infty}} \right) \norm{\eta}^{2}_{H^{1}}.
\end{align*}
Using Gronwall's inequality, we find that
\begin{align*}
\norm{\eta}_{H^{1}}^{2} 
&+ \frac{\ep}{2} \int_{0}^{t} \exp \left(  C(\ep, \beta,\delta) (t-s) + \int_{s}^{t} \frac{1}{\ep} \norm{\eta}^{2}_{L^{\infty}} \right) \norm{ \Delta \eta}^{2}_{H^{1}} ds \\
&\hskip 6cm \leq 
\norm{\eta_{0}}_{H^{1}}^{2} \exp\left(  C(\ep, \beta,\delta) t +  \int_{0}^{t} \frac{1}{\ep} \norm{\eta}^{2}_{L^{\infty}} \right) .
\end{align*}
From step 1, we already have that
$$\ep \int_{0}^{t} \norm{\eta}^{2}_{H^{2}} ds 
\leq 
C(\ep, \beta,\delta) \norm{\eta_{0}}_{L^{2}}^{2} e^t.$$ 
Using the Sobolev embedding, we get that 
$$\ep \int_{0}^{t} \norm{\eta}^{2}_{L^{\infty}} ds 
\leq 
C(\ep, \beta,\delta) \norm{\eta_{0}}_{L^{2}}^{2} e^t.$$
Inserting this into the inequality above, we obtain 
\begin{align*}
\norm{\eta}_{H^{1}}^{2} 
&+ \frac{\ep}{2} \int_{0}^{t} \exp \left(C(\ep, \beta,\delta) (t-s) + \int_{s}^{t} \frac{1}{\ep} \norm{\eta}^{2}_{L^{\infty}} \right) \norm{ \Delta \eta}^{2}_{H^{1}} ds \\
&\leq 
\norm{\eta_{0}}_{H^{1}}^{2} \exp\left( C(\ep, \beta,\delta)  t + \frac{1}{\ep^{2}} \norm{\eta_{0}}_{L^{2}}^{2} \exp\left(C(\ep, \beta,\delta)  t \right) \right) \\
&\leq \norm{\eta_{0}}_{H^{1}}^{2} \exp \left(\exp \left( C(\ep, \beta, \delta \norm{\eta_{0}}_{H^{1}} )t\right) \right) .
\end{align*}
Hence, we conclude that 
\begin{multline*}
\norm{\eta}_{H^{1}}^{2} 
+ \frac{\ep}{2} \int_{0}^{t} \exp  \left(\exp  \left(    C(\ep, \beta, \delta \norm{\eta_{0}}_{H^{1}} )(t-s)\right) \right) \norm{ \Delta \eta}^{2}_{H^{1}} ds \\
 \leq \norm{\eta_{0}}_{H^{1}}^{2} \exp \left(\exp \left( C(\ep, \beta, \delta \norm{\eta_{0}}_{H^{1}} )t\right) \right).
\end{multline*}
In particular, 
\begin{align*}
\norm{\eta}_{H^{1}}^{2} 
+ \frac{\ep}{2} \int_{0}^{t} \norm{ \Delta \eta}^{2}_{H^{1}} ds  \leq \norm{\eta_{0}}_{H^{1}}^{2} \exp \left(\exp \left( C(\ep, \beta, \delta \norm{\eta_{0}}_{H^{1}} )t\right) \right).
\end{align*}
\textbf{Step 3 : $H^{2}$ estimate.} We multiply \eqref{KStwo-dimension} by $\Delta^{2} \eta$ and integrate by parts to obtain that
\begin{align*}
\frac{1}{2}\frac{d}{dt} \norm{\eta}^{2}_{H^{2}} 
&= (\beta-1) \norm{\eta_{x}}^{2}_{H^{2}} - \norm{\eta_{y}}_{H^{2}}^{2} + \delta \norm{\Lambda^{\frac{3}{2}}\eta}^{2}_{H^{2}} - \ep \norm{\Delta \eta}^{2}_{H^{2}}+ \norm{\eta}_{L^{\infty}} \norm{\eta}_{H^{1}} \norm{\Delta^{2}\eta}_{L^{2}}.
\end{align*}
Using the same method as in step 1 and step 2, we can obtain 
\begin{align*}
\frac{d}{dt} \norm{\eta}^{2}_{H^{2}} & \leq -\frac{\ep}{2} \norm{\Delta \eta}^{2}_{H^{2}}+ \left(C(\ep, \beta,\delta) +\frac{1}{\ep}  \norm{\eta}^{2}_{L^{\infty}} \right)\norm{\eta}^{2}_{H^{1}}\\
& \leq -\frac{\ep}{2} \norm{\Delta \eta}^{2}_{H^{2}}+ \left(C(\ep, \beta,\delta) +\frac{1}{\ep}  \norm{\eta}^{2}_{L^{\infty}} \right)\norm{\eta}^{2}_{H^{2}},
\end{align*}
Using Gronwall's inequality again, we obtain that 
\begin{multline*}
\norm{\eta}_{H^{2}}^{2} 
+ \frac{\ep}{2} \int_{0}^{t} \exp \left( C(\ep, \beta,\delta) (t-s) + \int_{s}^{t} \frac{1}{\ep} \norm{\eta}^{2}_{L^{\infty}} \right) \norm{ \Delta \eta}^{2}_{H^{2}} ds \\
\leq 
\norm{\eta_{0}}_{H^{2}}^{2} \exp\left(C(\ep, \beta,\delta) t +  \int_{0}^{t} \frac{1}{\ep} \norm{\eta}^{2}_{L^{\infty}} \right).
\end{multline*}
Then, the Sobolev embedding implies that
\begin{multline*}
\norm{\eta}_{H^{2}}^{2} 
+ \frac{\ep}{2} \int_{0}^{t} \exp \left( \exp \left(   C(\ep, \beta, \delta, \norm{\eta_{0}}_{H^{2}} )(t-s)\right) \right) \norm{ \Delta \eta}^{2}_{H^{2}} ds \\
 \leq \norm{\eta_{0}}_{H^{2}}^{2} \exp \left( \exp \left(  C( \ep, \beta, \delta, \norm{\eta_{0}}_{H^{2}} )t\right) \right) .
\end{multline*}
Finally, we have 
\begin{align*}
\norm{\eta}_{H^{2}}^{2} 
+ \frac{\ep}{2} \int_{0}^{t} \norm{ \Delta \eta}^{2}_{H^{2}} ds  \leq \norm{\eta_{0}}_{H^{2}}^{2} \exp \left( \exp \left( C(\ep, \beta, \delta, \norm{\eta_{0}}_{H^{2}} )t\right) \right).
\end{align*}

\textbf{Step 4 : Existence of solution} We consider a positive, symmetric mollifier $J_{\ep'}$ (such as the periodic heat kernel), to approximate the initial value problem \eqref{KStwo-dimension} by the regularized problem
\begin{multline*}
\pa_{t} \eta_{\ep'}
+ J_{\ep'} * \frac{\pa_{x} \left(\left(J_{\ep'}*\eta_{\ep'}\right)^{2}\right)}{2}\\
 =J_{\ep'} * \left( (1-\beta)  \pa_{xx} (J_{\ep'}*\eta_{\ep'}) + \pa_{yy} (J_{\ep'}*\eta_{\ep'}) + \delta \Lambda^{3} (J_{\ep'}*\eta_{\ep'})- \ep \Delta^{2} (J_{\ep'}*\eta_{\ep'}) \right),
\end{multline*}
with initial data 
$$\eta_{\ep'}(0)=J_{\ep'}*\eta_{0} .$$

By the Picard's theorem, these sequence of regularized problems have a unique solution $\eta_{\ep'} = C^1 ([0,T_{\ep'}], H^{2}(\TT^{2}))$. Moreover, these problems verify the same energy estimates as in steps 1-3 and, as a consequence, we can take $T = T(\eta_{0})$ independent of $\ep'$. Passing to the limit we conclude the existence of at least one solution in 
$$
\eta\in L^\infty(0,T;H^2)\cap L^2(0,T;H^4)
$$
and since we have \textit{a priori} estimates, these solutions exist for arbitrary long time $T$.  

\textbf{Step 5 : Uniqueness.} 
We can prove the uniqueness of the solutions by contradiction, i.e. assuming that there exists two solutions of the problem \eqref{KStwo-dimension}, $\eta_{1}$ and $\eta_{2}$, corresponding to the same initial data $\eta_{0}$. We denote their difference by $\underline{\eta}$. Then we have 
\begin{equation}\label{KS-difference}
\underline{\eta}_{t} + \frac{1}{2} \left(\eta_{1}^{2}-\eta_{2}^{2}\right)_{x} + (\beta-1)  \underline{\eta}_{xx}-\underline{\eta}_{yy}- \delta \Lambda^{3}(\underline{\eta})+ \ep \Delta^{2} \underline{\eta} = 0. 
\end{equation}
As the proof of $L^{2}$ estimate in step 1, we multiply \eqref{KS-difference} by $\underline{\eta}$ and integrate by parts: 
\begin{align*}
\frac{1}{2} \frac{d}{dt} \norm{\underline{\eta}}^{2}_{L^{2}} 
&=  
(\beta-1) \norm{\underline{\eta}_{x}}^{2}_{L^{2}} - \norm{\underline{\eta}_{y}}_{L^{2}}^{2} + \delta \norm{\Lambda^{\frac{3}{2}}(\underline{\eta})}_{L^{2}} - \ep \norm{\Delta \underline{\eta}}^{2}_{L^{2}}- \int_{\TT^{2}} \frac{1}{2}  (\eta_{1}^{2}-\eta_{2}^{2})_{x} \underline{\eta} \\
& \leq C(\ep, \beta,\delta) \norm{\underline{\eta}}^{2}_{L^{2}} -\frac{\ep}{4} \norm{\Delta \underline{\eta}}^{2}_{L^{2}} + \frac{1}{2} \norm{\eta_{1}+\eta_{2}}_{L^{\infty}}  \norm{\underline{\eta}}_{H^{1}} 
\norm{\underline{\eta}}_{L^{2}}  \\
&\leq C(\ep, \beta,\delta)\norm{\underline{\eta}}^{2}_{L^{2}} -\frac{\ep}{4} \norm{\Delta \underline{\eta}}^{2}_{L^{2}} + \frac{\ep}{4}  \norm{\underline{\eta}}_{H^{1}}^{2} + \frac{1}{4 \ep}
\norm{\eta_{1}+\eta_{2}}_{L^{\infty}}^{2}  \norm{\underline{\eta}}_{L^{2}}^{2}\\
&\leq  \left( C(\ep, \beta,\delta) + \frac{1}{4 \ep} \norm{\eta_{1}+\eta_{2}}_{L^{\infty}}^{2} \right) \norm{\underline{\eta}}^{2}_{L^{2}} \\
&\leq \left(C(\ep, \beta,\delta) + \frac{1}{2 \ep} \left( \norm{\eta_{1}}_{L^{\infty}}^{2}+ \norm{\eta_{2}}_{L^{\infty}}^{2} \right) \right) \norm{\underline{\eta}}^{2}_{L^{2}} 
\end{align*}
Using Gronwall's inequality, we have that
\begin{align*}
\norm{\underline{\eta}}^{2}_{L^{2}} 
\leq \norm{\underline{\eta}_{0}}^{2}_{L^{2}} \exp \left( C(\ep, \beta,\delta) t + \frac{1}{ \ep} \int_{0}^{t} \norm{\eta_{1}}_{L^{\infty}}^{2}+ \norm{\eta_{2}}_{L^{\infty}}^{2} \right) .
\end{align*}
From step 1, we already have that
$$\ep \int_{0}^{t} \norm{\eta_{1}}^{2}_{H^{2}} ds 
\leq 
\norm{\eta_{0}}_{L^{2}}^{2} \exp\left(  C(\ep, \beta,\delta)t \right),$$
and 
$$\ep \int_{0}^{t} \norm{\eta_{2}}^{2}_{H^{2}} ds 
\leq 
\norm{\eta_{0}}_{L^{2}}^{2} \exp\left(  C(\ep, \beta,\delta)t \right).$$
Thus, the uniqueness of solution follows from the inequality
\begin{align*}
\norm{\underline{\eta}}^{2}_{L^{2}} 
\leq \norm{\underline{\eta}_{0}}^{2}_{L^{2}} \exp \left(\exp \left( C(\delta, \ep, \beta, \norm{\eta_{0}}_{L^{2}} )t\right) \right).
\end{align*}

\textbf{Step 6 : Endpoint continuity in time.}
To conclude the endpoint continuity, we can perform a standard argument using the parabolic gain of regularity $L^2(0,T;H^4)$. Indeed, we can take $0<\sigma\ll 1$ as small as desired and there exists a $0<\sigma'<\sigma$ such that $u(\sigma')\in H^4$. Repeating the same argument as before, we find a solution
$$
\eta_\sigma\in L^\infty(\sigma',T;H^4)\cap C([\sigma',T],H^2).
$$
Because of the uniqueness of solution we obtain the continuity of the original solution
$$
\eta\in C((0,T],H^2).
$$
Finally, the continuity at the origin is a consequence of the energy estimates.
\end{proof}

\section{Large time dynamics}
The goal of this section is to prove uniform boundedness of solutions $\eta \in L^{\infty}([0, \infty); L^{2}(\TT^{2}))$. In other words, we establish the existence of an absorbing ball in $L^{2}$ by collecting global bounds showing the dissipative character of the equation. We start by proving the following two Gagliardo-Nirenberg inequalities
\begin{lem}{\label{lemmatempor}}For smooth enough periodic functions with zero mean, we have that the following two inequalities hold true
$$\norm{\na \eta}_{L^{4}(\TT^{2})}^{2} \leq C \norm{\eta}_{L^\infty(\TT^{2})} \norm{ \Delta \eta}_{L^{2}(\TT^{2})},$$
$$\norm{\Delta \eta^{2}}_{L^{2}(\TT^{2})} \leq C \norm{\eta}_{L^{\infty}(\TT^{2})} \norm{\Delta \eta }_{L^{2}(\TT^{2})}.$$\end{lem}
\begin{proof}[proof]
We start proving the first inequality:
\begin{align*}
\norm{\na \eta}_{L^{4}(\TT^{2})}^{4} 
&= \int_{\TT^{2}} \left(\na \eta \cdot \na \eta\right)^{2} \\
&= - \int_{\TT^{2}} \eta \na\cdot \left(\na \eta |\na \eta|^{2}\right)\\ 
&= - \int_{\TT^{2}} \eta \Delta \eta |\na \eta|^{2} + \eta \na \eta \cdot \na |\na \eta|^{2} \\
& \leq C \norm{\eta}_{L^\infty(\TT^{2})} \norm{\Delta\eta}_{L^{2}(\TT^{2})} \norm{ |\na \eta|^{2}}_{L^{2}(\TT^{2})}\\ 
& \leq C  \norm{\eta}_{L^\infty(\TT^{2})} \norm{ \Delta \eta}_{L^{2}(\TT^{2})} \norm{ \na \eta}_{L^{4}(\TT^{2})}^{2}.
\end{align*}
Therefore, we conclude our result by noticing that
\begin{align*}
\norm{\Delta \eta^{2}}_{L^{2}(\TT^{2})}
&\leq C\norm{\eta \Delta \eta + |\na \eta|^{2}}_{L^{2}(\TT^{2})}\\
&\leq C\norm{\eta \Delta \eta}_{L^{2}(\TT^{2})} + \norm{ |\na \eta|^{2}}_{L^{2}(\TT^{2})}\\
&\leq C\norm{\eta}_{L^{\infty}(\TT^{2})} \norm{\Delta \eta }_{L^{2}(\TT^{2})}+ \norm{\na \eta}_{L^{4}(\TT^{2})}^{2}\\
&\leq C \norm{\eta}_{L^{\infty}(\TT^{2})} \norm{\Delta \eta }_{L^{2}(\TT^{2})} .
\end{align*}
\end{proof}
\begin{remark} We observe that the previous constants $C$ can be computed explicitly.
\end{remark}

The rest of this section is devoted to prove that the solutions of problem \eqref{KStwo-dimension} remain uniformly bounded in $L^{2}$. The following background flow method was first used by Nicolaenko, Scheurer, $\&$ Temam \cite{nicolaenko1985some} and then improved by Collet, Eckmann, Epstein $\&$ Stubbe \cite{collet1993global}, Goodman \cite{goodman1994stability} and Bronski \& Gambill \cite{bronski2006uncertainty}. Before stating the main result of this section, let us first define the following subspace of $H^2(\TT^2)$ : 
\begin{equation*}
	H^2_{\text{od}} (\TT^2) = \{ \eta \in H^2(\TT^2) : - \eta (-x, y) = \eta(x,y), \forall (x,y) \in \TT^2 \}
\end{equation*}
In terms of the rigorous results, a global bound on the solution is given by the following theorem : 
\begin{teo}{\label{absorbingset}}
Let $\eta_{0}\in H^2_{\text{od}} (\TT^2) $. Then the solution $\eta$ of the initial-value problem \eqref{KStwo-dimension} satisfies 
\begin{align}\label{absorbinginequality}
\limsup_{t\to\infty} \norm{\eta(t)}_{L^2(\TT^{2})} \leq R_{\ep, \delta,\beta}.
\end{align} 
where $R_{\ep, \delta, \beta}$ depends on $\ep, \delta, \beta$. 
\end{teo}

\begin{proof}[proof]
The proof is based on the construction of a Lyapunov functional, $\mathcal{F}(t)$, such that
$$
\frac{d}{dt}\mathcal{F}(t)\leq 0,
$$
if 
$$
\mathcal{F}(t)\geq R_{\ep, \delta,\beta},
$$
i.e. implying the existence of an absorbing set in $L^{2}$.
We first let $\phi$ be a smooth, $2\pi$-periodic function, which we will choose later. Then, we multiply equation \eqref{KStwo-dimension} by $\eta-\phi$, and integrate by parts:
\begin{align*}
\int \int \left( (\eta-\phi)_{t} + \eta \eta_{x}+ (\beta-1)  \eta_{xx}-\eta_{yy}- \delta \Lambda^{3}(\eta)+ \ep \Delta^{2} \eta \right) (\eta-\phi)=0,
\end{align*}
thus,
\begin{multline*}
\frac{1}{2} \frac{d}{dt} \norm{\eta-\phi}^{2}_{L^{2}}
+ \frac{1}{2} \int_{\TT^{2}}  \phi _{x} \eta^{2} 
- (\beta-1)\norm{\eta_{x}}^{2}_{L^{2}}
+ \norm{\eta_{y}}^{2}_{L^{2}} 
+ \ep \norm{\Delta  \eta}^{2}_{L^{2}}\\
- \int_{\TT^{2}} (\beta-1) \eta_{xx} \phi 
+ \int_{\TT^{2}} \eta_{yy} \phi 
- \delta \int_{\TT^{2}} \Lambda^{3} \eta(\eta-\phi)
- \int_{\TT^{2}} \ep\Delta \eta \Delta  \phi  = 0.
\end{multline*}
For the term corresponding to the nonlocal self-adjoint operator $\Lambda$, we have that   
$$\int_{\TT^{2}}  \Lambda \eta \phi = \int_{\TT^{2}}  \eta \Lambda  \phi.$$
Hence, by the Young's inequality and the H\"older inequality, we have that,
\begin{align}\label{inequality}
\begin{split}
\frac{1}{2}\frac{d}{dt} \norm{\eta-\phi}^{2}_{L^{2}}
= 	- \int_{\TT^{2}} \frac{\phi_{x}}{2}  \eta^{2} 
+ (\beta-1)\norm{\eta_{x}}^{2}_{L^{2}}
- \norm{\eta_{y}}^{2}_{L^{2}} 
- \ep \norm{\Delta  \eta}^{2}_{L^{2}}
+ \delta \int_{\TT^{2}} \Lambda^{3} \eta \eta  
\\
\hskip .2cm +(1-\beta)\int_{\TT^{2}} \eta_{x} \phi_{x}
+ \int_{\TT^{2}} \eta_{y} \phi_{y}
- \delta \int_{\TT^{2}} \Lambda^{3} \eta\phi
+ \int_{\TT^{2}} \ep \Delta \eta \Delta  \phi \\
\quad \leq - \int_{\TT^{2}} \frac{\phi_{x}}{2}  \eta^{2} 
+ 2|\beta-1|\norm{\eta_{x}}^{2}_{L^{2}}
- \frac{1}{2} \norm{\eta_{y}}^{2}_{L^{2}} 
- \frac{\ep}{2} \norm{\Delta  \eta}^{2}_{L^{2}} 
+ \delta \norm{\Lambda^{\frac{3}{2}} \eta}_{L^2}^2 \\
\quad \hskip .5cm +\frac{1}{2} \norm{\eta}_{L^{2}}^{2} 
+|\beta-1| \norm{\phi_{x}}_{L^{2}}^{2}+\frac{1}{2}\norm{\phi_{y}}_{L^{2}}^{2}+\frac{\ep}{2}\norm{ \Delta \phi}_{L^{2}}^{2} + \frac{\delta^2}{2} \norm{\phi}_{H^{3}}^{2}
\end{split}
\end{align}
Now, we define the following function 
\begin{equation*}
f(|\xi|)= f(z): = \frac{1}{2} + 2(\beta + 2) z^{2} + \delta z^{3}- \frac{\ep}{4} z^{4} , \quad z \geq 0.
\end{equation*}
Observe that $f(z)$ has at most three zeros. Since $\ep > 0$, we find that $f(z)$ is bounded and has global maximum $\frac{1}{2} + (8(\beta+2)\ep + 4\delta^2)  \delta^2 /\ep^3$ at point $z_1 = 2 \delta / \ep$. Then we have the following inequality
\begin{align*}
\frac{1}{2} + 2|\beta -1 | |\xi_{1}|^{2} -\frac{1}{2} |\xi_{2}|^{2} + \delta|\xi|^{3}- \frac{\ep}{4} |\xi|^{4} 
& \leq f(z_1) = \frac{1}{2} + (8(\beta+2)\ep + 4\delta^2)  \delta^2 /\ep^3\\
&:= C(\beta,\delta,\ep).
\end{align*}
Inserting this relation into \eqref{inequality}, we obtain that  
\begin{align}{\label{equality1}}
\begin{split}
\frac{1}{2} \frac{d}{dt} \norm{\eta-\phi}^{2}_{L^{2}(\TT^{2})}
& = - \frac{\ep}{4} \norm{\Delta  \eta}^{2}_{L^{2}(\TT^{2})}
 - \norm{\eta}_{L^{2}}^{2} +  \int_{\TT^{2}} \left(\lambda - \frac{\phi_{x} }{2}\right) \eta^{2} 
+ F(\phi),
\end{split}
\end{align}
where 
$$\lambda = C(\beta, \delta, \ep)+ 1, $$
and
$$F(\phi)=|\beta-1| \norm{\phi_{x}}_{L^{2}}^{2}+\frac{1}{2}\norm{\phi_{y}}_{L^{2}}^{2}+\frac{\ep}{2}\norm{ \Delta \phi}_{L^{2}}^{2} + \frac{\delta^2}{2} \norm{\phi}_{H^{3}}^{2}. $$
Now, we choose $\phi (x,y)$ such that 
\begin{align*}%\label{functionphi}
\frac{\phi_x}{2} = - \lambda \sum_{0<|\xi_1| \leq A /\ep} e^{-i x \xi_1} ,
\end{align*}
which is possible since the right-hand side has zero horizontal mean value. Here, $A$ is a constant independent of $\ep$, which will be determined later.

Then we claim that 
\begin{align}\label{ineqphi}
|\int_{\TT^{2}} \left(\lambda - \frac{\phi_{x} }{2}\right) \eta^{2}|
\leq \frac{\ep}{8} \norm{\Delta \eta}_{L^2(\TT^2)}^2
\end{align}
Indeed, with the choice of $\phi$, we have that 
\begin{align*}%\label{phiphi}
\int_{\TT^{2}} \left(\lambda - \frac{\phi_{x}}{2}\right) \eta^{2}
&= \int_{\TT^{2}} \lambda \big(1 +\sum_{0<|\xi_1| \leq A /\ep} e^{-i x \xi_1}\big) \eta^{2}\\	
&= \lambda  \sum_{|\xi_1| \leq A /\ep} \int_{\TT^{2}}  e^{-i x \xi_1} \eta^{2}\\
&= 
\lambda \sum_{|\xi_1| \leq A /\ep} \widehat{\eta^2}(\xi_1, 0 ).
\end{align*}
Since the odd symmetry in the $x$-direction is preserved by the equation and $\eta_0 \in H^2_{\text{od}}(\TT^2)$, we have the fact that $\eta^2 (0, y, t)=0$ for any $y \in \TT$ and $t \geq 0$, so that
\begin{align*}%\label{fouriereta}
\sum_{\xi_1 \in \ZZ} \widehat{\eta^2}(\xi_1, 0 ) = 0,
\end{align*}
which implies
\begin{align*}%\label{fouriereta}
		\sum_{|\xi_1| \leq A /\ep} \widehat{\eta^2}(\xi_1, 0 ) = -\sum_{|\xi_{1}| > A/\ep} \widehat{\eta^2} (\xi_{1}, 0 ). 
\end{align*}
Then, by Cauchy-Schwartz, we bound
\begin{align*}
\left|	\sum_{|\xi_1| \leq A /\ep} \widehat{\eta^2}(\xi_1, 0 )\right|
&\leq  \sum_{|\xi_1|> A/\ep}|\xi_1|^2\left|\widehat{\eta^2}(\xi_1,0)\right|\frac{1}{|\xi_1|^2}\\
&\leq \frac{\ep}{A}\sum_{|\xi_1|> A/\ep}|\xi_1|^2\left|\widehat{\eta^2}(\xi_1,0)\right|\frac{1}{|\xi_1|}\\
&\leq \frac{\ep}{A} \Big(\sum_{|\xi_1|> A/\ep}|\xi_1|^4\left|\widehat{\eta^2}(\xi_1,0)\right|^2 \Big)^{\frac{1}{2}}
\Big(\sum_{\xi_{1} > A/\ep}\frac{1}{|\xi_1|^2}\Big)^{\frac{1}{2}}\\
& \leq \frac{C\ep}{A} \Big(\sum_{\xi_{2} \in \ZZ}\sum_{\xi_1\in \ZZ}\big(|\xi_1|^2 + |\xi_2|^2 \big)^2\left|\widehat{\eta^2}(\xi_1,\xi_2)\right|^2 \Big)^{\frac{1}{2}}\\
&\leq  \frac{C\ep}{A}\|\Delta(\eta^2)\|_{L^2(\TT^2)}
\leq  \frac{C\ep }{A}\|\Delta\eta\|_{L^2(\TT^2)}^2,
\end{align*}
where we have used Lemma \ref{lemmatempor} in the last step as well as the Sobolev embedding $H^2(\TT^2)$ into $L^{\infty}(\TT^2)$. Hence,
\begin{equation*}
	|\int_{\TT^{2}} \left(\lambda - \frac{\phi_{x} }{2}\right) \eta^{2}| \leq \lambda \frac{C \ep }{A}\|\Delta\eta\|_{L^2(\TT^2)}^2
\end{equation*}
The choice $A>8C \lambda$ justifes the claim \eqref{ineqphi}.\\

We consider the functional
$$
\mathcal{F}(t)=\norm{\eta-\phi}^{2}_{L^{2}(\TT^{2})}.
$$
Inserting this into \eqref{equality1}, we obtain
\begin{align*}
\begin{split}
\frac{1}{2} \frac{d}{dt} \norm{\eta-\phi}^{2}_{L^{2}(\TT^{2})}
& \leq  - \frac{\ep}{8} \norm{\Delta  \eta}^{2}_{L^{2}(\TT^{2})} - \norm{\eta}_{L^{2}}^{2} + F(\phi)
\end{split}
\end{align*}
Hence
\begin{align}\label{estimate before gronwall}
\frac{1}{2} \frac{d}{dt} \norm{\eta-\phi}^{2}_{L^{2}(\TT^{2})}
& \leq -\norm{\eta - \phi}_{L^{2}(\TT^{2})}^{2} + \norm{\phi}^{2}_{L^{2}} + F(\phi),
\end{align}
or, equivalently,
$$
\frac{1}{2} \frac{d}{dt}\mathcal{F} \leq -\mathcal{F} + \norm{\phi}^{2}_{L^{2}} + F(\phi),
$$
which allows us to conclude the uniform boundedness of $\mathcal{F}$. Indeed, using Gronwall inequality, we  immediately obtain 
\begin{align*}
\begin{split}
\norm{\eta-\phi}^{2}_{L^{2}(\TT^{2})} \leq \big(\norm{\eta_{0}-\phi}^{2}_{L^{2}(\TT^{2})} -\norm{\phi}^{2}_{L^{2}} - F(\phi) \big) e^{-2 t} +\norm{\phi}^{2}_{L^{2}} + F(\phi).
\end{split}
\end{align*}
where
$$F(\phi)=|\beta-1| \norm{\phi_{x}}_{L^{2}}^{2}+\frac{1}{2}\norm{\phi_{y}}_{L^{2}}^{2}+\frac{\ep}{2}\norm{ \Delta \phi}_{L^{2}}^{2} + \frac{\delta^2}{2} \norm{\phi}_{H^{3}}^{2}. $$
Thus, if $1\ll\|\eta(t)\|_{L^2}$, we conclude that 
\begin{align*}
\norm{\eta}_{L^{2}(\TT^{2})}
& \leq \norm{\eta-\phi}_{L^{2}(\TT^{2})} + \norm{\phi}_{L^{2}(\TT^{2})}\\
& \leq \left(\norm{\eta_{0}-\phi}^{2}_{L^{2}(\TT^{2})} +\norm{\phi}^{2}_{L^{2}} + F(\phi) \right)^{\frac{1}{2}} e^{-  t} +2 \norm{\phi}_{L^{2}} + F(\phi)^{\frac{1}{2}}\\
& \leq \left( \norm{\eta_{0}}_{L^{2}} +2\norm{\phi}_{L^{2}} + (\beta+2) \norm{\phi}_{H^{1}}+\sqrt{\ep}\norm{\phi}_{H^{2}}+ \delta \norm{\phi}_{H^{3}}
\right) e^{- t} \\
&\quad+2 \norm{\phi}_{L^{2}} + (\beta+2) \norm{\phi}_{H^{1}}+\sqrt{\ep}\norm{\phi}_{H^{2}}+ \delta \norm{\phi}_{H^{3}}\\
& := R_{\ep, \delta,\beta}.
\end{align*}
This completes the proof of Theorem \ref{absorbingset}.
\end{proof}

Similarly as we have obtained that there exists an absorbing set in $L^{2}$, we can conclude the existence of an absorbing set in higher Sobolev norms. 
\begin{teo}{\label{absorbingsetH4}}
Let $\eta_{0}\in H^2_{\text{od}}(\TT^2)$. Then the solution $\eta$ of the initial-value problem \eqref{KStwo-dimension} satisfies 
$$ \limsup_{t\to\infty} \norm{\eta(t)}_{H^2(\TT^{2})} \leq R'_{\ep, \delta,\beta}. 
$$
where $R'_{\ep, \delta, \beta}$ is a constant depending on $\ep, \delta, \beta$. 
\end{teo}
\begin{proof}[proof]
Recalling the existence of an absorbing set in the $L^{2}$-norm and the regularity results in Theorem \ref{globalexistence}, so the proof is straightforword by using a bootstrap argument.

We first show that there exists an absorbing set in the $H^{1}$-norm. Inequality (\ref{absorbinginequality}) implies that for a $T >0$ large enough, we have
$$ \norm{\eta(t)}_{L^2(\TT^{2})} \leq R_{\ep, \beta, \delta}+1 ,\;\; \forall t > T.$$
Combining this inequality with the $L^{2}$ energy estimate in the proof of Theorem \ref{globalexistence}, we obtain that
$$\norm{\eta}_{L^{\infty}([0,T];L^{2})}^{2} \leq \norm{\eta_{0}}_{L^{2}}^{2} \exp \left( C(\ep, \beta,\delta) T \right), $$ 
which results in that there exists a constant depending on initial data, $\delta$ and $\epsilon$ such that 
$$ \max_{0\leq t<\infty} \norm{\eta(t)}_{L^{2}(\TT^{2})}^{2} \leq C (\norm{\eta_{0}}_{L^{2}},\ep, \beta, \delta).$$
We multiply \eqref{KStwo-dimension} by $- \Delta \eta$ and integrate by parts to obtain that
\begin{align*}
\frac{1}{2}\frac{d}{dt} \norm{\eta}^{2}_{H^{1}} 
&\leq (\beta-1) \norm{\eta_{x}}^{2}_{H^{1}} - \norm{\eta_{y}}_{H^{1}}^{2} + \delta \norm{\Lambda^{\frac{3}{2}}\eta}^{2}_{H^{1}} - \ep \norm{\Delta \eta}^{2}_{H^{1}}+ \norm{\eta}_{L^{4}}^2 \norm{\Delta \eta_x}_{L^{2}}\\
&\leq c_\ep \|\eta\|_{L^4}^4-\frac{\ep}{2} \norm{ \eta}^{2}_{H^{3}}+|\beta-1| \norm{\eta_{x}}^{2}_{H^{1}} - \norm{\eta_{y}}_{H^{1}}^{2} + \delta \norm{\Lambda^{\frac{3}{2}}\eta}^{2}_{H^{1}}\\
&\leq C_{\ep,\delta,\eta_0}\|\eta\|_{H^1}^2-\frac{\ep}{2} \norm{ \eta}^{2}_{H^{3}}+|\beta-1| \norm{\eta_{x}}^{2}_{H^{1}} - \norm{\eta_{y}}_{H^{1}}^{2} + \delta \norm{\Lambda^{\frac{3}{2}}\eta}^{2}_{H^{1}},\\
&\leq C_{\ep, \beta, \delta,\eta_0}\|\eta\|_{H^1}^2-\frac{\ep}{4} \norm{ \eta}^{2}_{H^{3}}+C_{\ep,\delta}\|\eta\|_{L^2}^2\\
&\leq -\frac{\ep}{8} \norm{ \eta}^{2}_{H^{3}}+C_{\ep,\beta,\delta,\eta_0}\\
&\leq -\frac{\ep}{16} \norm{ \eta}^{2}_{H^{1}}-\frac{\ep}{16} \norm{ \eta}^{2}_{H^{3}}+C_{\ep, \beta,\delta,\eta_0}.
\end{align*}
where we used the Plancherel Theorem, the Poincar\'e inequality and the Sobolev inequality
$$
\|\eta\|_{L^4}\leq C\|\eta\|_{H^{1/2}}\leq C\|\eta\|_{L^2}^{1/2}\|\eta\|_{H^1}^{1/2}.
$$
It follows that 
\begin{align*}
\frac{d}{dt} \norm{\eta}^{2}_{H^{1}} + \frac{\ep}{8} \norm{ \eta}^{2}_{H^{1}}
& \leq C (\norm{\eta_{0}}_{L^{2}},\ep,\beta, \delta).
\end{align*}
Using the Gronwall inequality, we immediately obtain the uniform bound 
\begin{align*}
\norm{\eta(t)}^{2}_{H^{1}}
& \leq  C (\norm{\eta_{0}}_{H^{1}},\ep, \beta, \delta).
\end{align*}

Recall that the $H^2$ energy estimate is 
\begin{align*}
\frac{d}{dt} \norm{\eta}^{2}_{H^{2}} + \frac{\ep}{2} \norm{\eta}^{2}_{H^{4}}
& \leq \left( C(\ep, \beta,\delta)+\frac{1}{\ep}  \norm{\eta}^{2}_{L^{\infty}} \right) \norm{\eta}^{2}_{H^{1}},
\end{align*}
so we can mimic the previous proof to obtain that there exists an absorbing set in $H^{2}$.
The proof is completed.

\end{proof}

\begin{remark}
	It is worth to point out that in the two theorems above, we show the large time dynamic of the equation for initial data belonging to $	H^2_{\text{od}}(\TT^2)$, which requires an odd symmetry condition in $x$-direction. Since the odd symmetry in $x$-direction is preserved by equation \eqref{KStwo-dimension}, working in  $H^2_{\text{od}}(\TT^2)$ is reasonable.
	To remove this condition, we require more regularity on the test function $\phi$ and this causes the case to become more delicate to handle. This case will be shown in our upcoming work. 
\end{remark}

\section{Analyticity}
The aim of this section is to show instant analyticity for the solutions of \eqref{KStwo-dimension}. We shall prove that the solutions of \eqref{KStwo-dimension} are analytic in a strip. In order to do this, we use the method developed by Collet, Eckmann, Epstein $\&$ Stubbe in \cite{collet1993analyticity} (see also \cite{foias1989gevrey}). Roughly speaking, our proof is based on \textit{a priori} estimates for functions in certain Gevrey class.

Given a function $\sigma(t)$ positive (see its formula explicit below), we consider the weighted exponential operators
$$e^{\sigma(t) \Lambda} \eta = \sum_{\xi\in \TT^{2}} \hat{\eta}(\xi) e^{\sigma(t) |\xi| } e^{i \xi \cdot x} $$
for functions in the space
$$ G := \{ \eta \in L^{2}(\TT^{2}):  \sum_{\xi \in \ZZ^{2}} e^{2\sigma(t) |\xi|} |\hat{\eta}(\xi)|^{2} < \infty \}.$$
We observe that the functions in $G$ are analytic. We also define the inner product and norm on this Hilbert space by 
$$
\langle \mu , \eta \rangle_{\sigma(t)} = \int_{\mathbb{T}^2}e^{\sigma(t) \Lambda} \mu \overline{e^{\sigma(t) \Lambda} \eta} =4\pi^2\sum_{\ell\in\mathbb{Z}^2}e^{2\sigma(t) |\ell|}\widehat{\mu(\ell)}\overline{\widehat{\eta(\ell)}},
$$
$$
\norm{\eta}^{2}_{\sigma(t)}= \norm{\etl\eta}^{2}_{L^{2}}.
$$

With these previous definitions, we can state the main result of this section.
\begin{teo}\label{analyticteo}
Let $\eta_{0}$ be given in $H^{2}_{\text{od}}(\TT^{2})$. Then, there exists $T_0$ depending on $\eta_{0}, \ep, \beta, \delta$ such that the solution of \eqref{KStwo-dimension} satisfies
$$\norm{e^{\sigma(t) \Lambda} \eta (t)}_{L^{2}}^{2} \leq  1+2 C_{\ep,\beta,\delta,\eta_0}^{2}, \;\forall\;t > 0$$
where $\sigma(t)= \min\{ \tanh(t), \tanh\left(\frac{T_0}{2}\right)\} $. In particular, it becomes analytic for $t>0$.
\end{teo}
Before proving theorem \ref{analyticteo}, we first state some auxiliary lemmas:
\begin{lem}{\label{estimation}}
For every $b > a \geq 0$,
$$ \norm{\Lambda^{\frac{a}{2}} \eta }^{2}_{\sigma(t)} \leq
\norm{\Lambda^{\frac{b}{2}} \eta}^{\frac{2 a}{b}}_{\sigma(t)} \norm{\eta}^{2-\frac{2a}{b}}_{\sigma(t)}  . $$
\end{lem}
\begin{proof}[proof]
\begin{align*}
\norm{\Lambda^{\frac{a}{2}} \eta }^{2}_{\sigma(t)} 
&=\norm{\etl \Lambda^{\frac{a}{2}} \eta}^{2}_{L^{2}}\\
&= \sum_{\xi \in \ZZ^{2}} e^{2\sigma(t) |\xi|} |\xi|^{a} |\hat{\eta}(\xi)|^{2}\\
&= \sum_{\xi \in \ZZ^{2}} e^{\sigma(t) |\xi| \frac{2a}{b}} |\xi|^{a} |\hat{\eta}(\xi)|^{\frac{2 a}{b}} e^{\sigma(t) |\xi| (2-\frac{2a}{b})} |\hat{\eta}(\xi)|^{2-\frac{2 a}{b}} \\
&\leq \left(\sum_{\xi \in \ZZ^{2}} \left(e^{\sigma(t) |\xi| \frac{2a}{b}} |\xi|^{a} |\hat{\eta}(\xi)|^{\frac{2 a}{b}}\right)^{\frac{b}{a}} \right)^{\frac{a}{b}}  \left( \sum_{\xi \in \ZZ^{2}} \left(e^{\sigma(t) |\xi| (2-\frac{2a}{b})} |\hat{\eta}(\xi)|^{2-\frac{2 a}{b}}\right)^{\frac{b}{b-a}} \right)^{\frac{b-a}{b}}\\
&\leq \left(\sum_{\xi \in \ZZ^{2}} e^{2\sigma(t) |\xi|} |\xi|^{b} |\hat{\eta}(\xi)|^{2} \right)^{\frac{a}{b}} 
\left(\sum_{\xi \in \ZZ^{2}} e^{2\sigma(t) |\xi|} |\hat{\eta}(\xi)|^{2} \right)^{\frac{b-a}{b}}\\
&\leq \norm{\Lambda^{\frac{b}{2}} \eta}_{\sigma(t)}^{\frac{2 a}{b}} \norm{\eta}_{\sigma(t)}^{2-\frac{2a}{b}}.
\end{align*}
\end{proof}
And an auxiliary lemma estimating the nonlinear term:
\begin{lem}\label{nonlinearterm}
$
| \langle \eta\eta_{x} , \eta \rangle_{\sigma(t)} | \leq c \norm{\Lambda \eta }_{\sigma(t)} \norm{ \Lambda^{\frac{1}{2}} \eta}_{\sigma(t)}^{2}.
$
\end{lem}
\begin{proof}[proof]
We first denote $\eta^{*}= \etl \eta $, then 
$$ \widehat{\eta^{*}(j)}= e^{\sigma(t) |j|} \hat{\eta}(j). $$
By the definition of Fourier series, we have 
$$
\eta = \sum_{j \in \ZZ^{2}} \widehat{\eta}(j) e^{ij\cdot x},\;\;\; \etl\eta = \sum_{j \in \ZZ^{2}} \widehat{\eta}(j) e^{ij\cdot x} e^{\sigma(t) |j|}.
$$
In fact,
\begin{align*}
\langle \eta\eta_{x} , \eta \rangle_{\sigma(t)} 
& = (2 \pi)^{2} i \sum_{\ell\in\mathbb{Z}^2}\sum_{j\in\mathbb{Z}^2} (\widehat{\eta}(j) \widehat{\eta}(\ell-j)  j_1 \overline{\widehat{\eta}}(\ell) e^{ 2 \sigma(t) |\ell|}\\
&= (2 \pi)^{2} i \sum_{\ell\in\mathbb{Z}^2}\sum_{j\in\mathbb{Z}^2} \left(\widehat{\eta^{*}}(j) \widehat{\eta^{*}}(\ell-j)\right) j_1 \overline{\widehat{\eta^{*}}}(\ell) e^{ \sigma(t) (|\ell| - |j|-|\ell-j|)}.
%\\
%&=\int_{\TT^{2}} e^{-\sigma(t) \Lambda }\eta^* e^{-\sigma(t) \Lambda }\eta_{x}^* \overline{e^{ \sigma(t) \Lambda }\eta^{*}}dx
\end{align*}
Since $|\ell|\leq |j| + |\ell-j|$ leads to $|\ell| - |j|-|\ell-j| \leq 0$, we have that
$$ e^{ \sigma(t)(|\ell| - |j|-|\ell-j|)} \leq 1.$$
Moreover,
\begin{align*}
\big{|} \langle \eta\eta_{x} , \eta \rangle_{\sigma(t)}  \big{|} 
&\leq (2 \pi)^{2} \bigg{|}i \sum_{\ell\in\mathbb{Z}^2}\sum_{j\in\mathbb{Z}^2} \left(\widehat{\eta^{*}}(j) \widehat{\eta^{*}}(\ell-j)\right) j_1 \overline{\widehat{\eta^{*}}}(\ell) e^{ \sigma(t) (|\ell| - |j|-|\ell-j|)}\bigg{|}\\
&\leq(2 \pi)^{2} \sum_{\ell\in\mathbb{Z}^2}\sum_{j\in\mathbb{Z}^2} |\widehat{\eta^{*}}(j)| |\widehat{\eta^{*}}(\ell-j)| |j| |\overline{\widehat{\eta^{*}}}(\ell)|\\
&  = (2 \pi)^{2} \sum_{k+j = \ell }  |\widehat{\eta^{*}}(j)| |j| |\widehat{\eta^{*}}(k)|  |\overline{\widehat{\eta^{*}}}(\ell)|\\
& = \int_{\mathbb{T}^2} \phi(x) \theta(x) \overline{\psi}(x) dx
\end{align*}
 where 
\begin{align*}
 \phi(x) = \sum_{j \in \mathbb{Z}^2} |\widehat{\eta^{*}}(j)| |j| e^{ijx}, \quad
 \theta(x) = \sum_{k \in \mathbb{Z}^2} |\widehat{\eta^{*}}(k)| e^{ikx}, \quad
 \psi(x) = \sum_{\ell \in \mathbb{Z}^2} |\widehat{\eta^{*}}(\ell)| e^{i \ell x} .
\end{align*}
Notice that $|\hat{\phi}(j)| = |j| |\hat{\eta^{*}} (j)|$ and $|\hat{\theta}(k)| = |\hat{\eta^{*}} (k)|$. We now bound the last integral by the H\"older inequality, 
\begin{align*}
  \int_{\mathbb{T}^2} \phi(x) \theta(x) \overline{\psi}(x) dx
  & \leq \norm{\phi(x)}_{L^2}  \norm{\theta(x)}_{L^4}  \norm{\psi(x)}_{L^4}\\
  & \leq c \norm{\phi(x)}_{L^2} \norm{\theta(x)}^2_{H^{\frac{1}{2}}}\\
  & =  c \left(\sum_{j \in \mathbb{Z}^2} |\hat{\phi}(j)|^2 \right)^{\frac{1}{2}} \sum_{j \in \mathbb{Z}^2}  \left(|k|^{\frac{1}{2}} |\hat{\theta}(k)|\right)^2 \\
  & =  c \left(\sum_{j \in \mathbb{Z}^2} |j|^2 |\hat{\eta^{*}} (j)|^2 \right)^{\frac{1}{2}} \sum_{j \in \mathbb{Z}^2}  \left(|k|^{\frac{1}{2}} |\hat{\eta^{*}} (k)|\right)^2 \\
  & = c \norm{\Lambda \eta^{*}}_{L^2} \norm{\Lambda^{\frac{1}{2}} \eta^{*}}^2_{L^2}\\ 
   & \leq c\norm{ \Lambda^{\frac{1}{2}} \eta}_{\sigma(t)}^{2} \norm{\Lambda \eta}_{\sigma(t)}
\end{align*}
where we used the Sobolev embedding $H^{\frac{1}{2}}(\mathbb{T}^2) \hookrightarrow L^4(\mathbb{T}^2) $ and the Plancherel theorem in the computation above.
\end{proof}
We begin now the proof of theorem \ref{analyticteo}:
\begin{proof}[Proof of Theorem \ref{analyticteo}]
We first take inner product of \eqref{KStwo-dimension} with $\eta(t)$ in Gevrey class $G$,
\begin{equation}\label{KSinnerpro}
\langle \eta_{t} , \eta \rangle_{\sigma(t)} + \langle \eta \eta_{x}, \eta \rangle_{\sigma(t)}  + \langle (\beta-1)  \eta_{xx}-\eta_{yy}- \delta \Lambda^{3}(\eta)+ \ep \Delta^{2} \eta , \eta \rangle_{\sigma(t)} = 0
\end{equation}
Note that 
$$
\frac{1}{2} \frac{d}{dt} \langle \eta , \eta \rangle_{\sigma(t)} = \sigma'(t) \Lambda \langle \eta , \eta \rangle_{\sigma(t)} + \langle \eta_{t} , \eta \rangle_{\sigma(t)}
$$
then we have 
$$\langle \eta_{t} , \eta \rangle_{\sigma(t)}= \frac{1}{2} \frac{d}{dt} \langle \eta , \eta \rangle_{\sigma(t)} - \sigma'(t) \Lambda \langle \eta , \eta \rangle_{\sigma(t)}. $$
Substituting this into \eqref{KSinnerpro}, 
\begin{align*}
\frac{1}{2} \frac{d}{dt} \langle \eta , \eta \rangle_{\sigma(t)} 
&= \sigma'(t) \langle \Lambda \eta , \eta \rangle_{\sigma(t)}- \langle (\beta-1)  \eta_{xx}-\eta_{yy}- \delta \Lambda^{3}(\eta)+ \ep \Delta^{2} \eta , \eta \rangle_{\sigma(t)}  - \langle \eta \eta_{x}, \eta \rangle_{\sigma(t)} \\
&\leq \sigma'(t) \norm{\Lambda^{\frac{1}{2}} \eta }_{\sigma(t)}^{2} + \beta \norm{\Lambda \eta }_{\sigma(t)}^{2} +\delta \norm{\Lambda^{\frac{3}{2}} \eta }_{\sigma(t)}^{2}- \ep \norm{\Lambda^{2} \eta }_{\sigma(t)}^{2} + c \norm{ \Lambda \eta}_{\sigma(t)} \norm{\Lambda^{\frac{1}{2}} \eta}_{\sigma(t)}^{2} \\
&\leq \sigma'(t) \norm{\Lambda^{2} \eta }_{\sigma(t)}^{\frac{1}{2}} \norm{\eta }_{\sigma(t)}^{\frac{3}{2}} + \beta  \norm{\Lambda^{2} \eta }_{\sigma(t)} \norm{\eta }_{\sigma(t)} +\delta  \norm{\Lambda^{2} \eta }_{\sigma(t)}^{\frac{3}{2}} \norm{\eta }_{\sigma(t)}^{\frac{1}{2}}\\
&\;\;- \ep \norm{\Lambda^{2} \eta }_{\sigma(t)}^{2}+ c \norm{ \Lambda^{2} \eta}_{\sigma(t)} \norm{\eta}_{\sigma(t)}^{2}.
\end{align*}
By the Young inequality, we have 
\begin{equation}{\label{eta-sigma-estimation}}
\frac{d}{dt} \norm{\eta }_{\sigma(t)}^{2} 
\leq \left(D_{1} \left(\sigma'(t)\right)^{\frac{4}{3}} + D_{2} + D_{3} \right) \norm{\eta }_{\sigma(t)}^{2}+ D_{4} \norm{\eta}_{\sigma(t)}^{4}
 \end{equation}
where $D_{1} =(\frac{3}{2\ep})^{1/3}$, $D_{2} = \frac{9}{4\ep} \beta^{2}$, $D_{3}= (\frac{27}{4 \ep})^{3} \delta^{4} $, $D_{4}= \frac{9}{4 \ep} c^{2}$.

By the definition of $\sigma(t)= \min\{ \tanh(t), \tanh\left(\frac{T_0}{2}\right)\} $, we have 
$$\sigma'(t) \leq 1.$$
Inserting this into \eqref{eta-sigma-estimation}, we obtain 
\begin{equation*}
\frac{d}{dt} \left(1+ \norm{\eta }_{\sigma(t)}^{2} \right)
\leq K \left(1+ \norm{\eta }_{\sigma(t)}^{2}  \right)^{2}
\end{equation*}
with 
$$K = D_{1} + D_{2} + D_{3} + D_{4}.$$

Define $y_{1}(t) = 1+ \norm{\eta }_{\sigma(t)}^{2}$, it turns out to deal with the following ordinary differential inequality
    \begin{equation*}
    \left\{
    \begin{aligned}
   & y'_{1}(t) \leq K y_{1}(t)^{2} \\
   & y_{1}(0) =  1+ \norm{\eta_{0}}_{L^{2}}^{2} .
    \end{aligned}
    \right.
    \end{equation*} 
After solving this ODI, we obtain 
\begin{equation*}
\norm{\eta (t)}_{\sigma(t)}^{2} \leq  1+2 \norm{\eta_{0}}_{L^{2}}^{2}\;\text{for}\; t \in (0,T_0],
\end{equation*}
where
$$
T_0 = \frac{1}{2K \left(1+ \norm{\eta_{0}}_{L^{2}}^{2}\right)}.
$$

Moreover, according to theorem \ref{globalexistence} and  theorem \ref{absorbingset}, we know that the solution is unique, global in time and stays in a ball of Radius $R_{\ep,\beta,\delta}$ once it has entered it, that is to say, $\norm{\eta(t)}_{L^{2}}$ remains bounded for all time,
$$ \limsup_{t\to\infty} \norm{\eta(t)}_{L^2(\TT^{2})} \leq R_{\ep,\beta, \delta}. $$
So up to now, we already prove local analyticity of $\eta(t)$
\begin{equation}\label{localanalytic}
\norm{\eta (t)}_{\sigma(t)}^{2} \leq  1+2 C_{\ep, \beta,\delta,\eta_0}^{2} \;\text{for}\; t \in (0,T_0].
\end{equation}
In order to obtain global analyticity, we follow the previous idea and repeat the argument above starting at $\frac{T_0}{2}$. We consider time $t \in [\frac{T_0}{2},\frac{3T_0}{2}]$, and let $y_{2} (t)=  1+ \norm{\eta(t) }_{\sigma(t-\frac{T_0}{2})}^{2}$, so $y_{2}\left(\frac{T_0}{2}\right)=  1+ \norm{\eta\left(\frac{T_0}{2}\right)}_{L^{2}}^{2}$. Thus, solving the following ordinary differential inequality 
    \begin{equation*}
    \left\{
    \begin{aligned}
   & y'_{2}(t) \leq K y_{2}(t)^{2} \\
   & y_{2}\left(\frac{T_0}{2}\right) = 1+ \left\|\eta\left(\frac{T_0}{2}\right)\right\|_{L^{2}}^{2},
    \end{aligned}
    \right.
    \end{equation*} 
we have that,
\begin{align*}
\norm{\eta (t)}_{\sigma\left(t-\frac{T_0}{2}\right)}^{2} 
\leq  1+2 \left\|\eta\left(\frac{T_0}{2}\right)\right\|_{L^{2}}^{2}
\leq  1+ 2 C_{\ep,\beta,\delta,\eta_0}^{2},
\end{align*}
for time $t \in [\frac{T_0}{2},\frac{3T_0}{2}]$.

By the definition of $\sigma(t)$ and observe that $ \tanh(t)$ is strictly increasing, we know that $\sigma(t)$ remains being a constant after time  $\frac{T_0}{2}$, and this constant is $\sigma\left(\frac{T_0}{2}\right)= \tanh\left(\frac{T_0}{2}\right)$, so we choose $t=T_0$ in the inequality above, then 
$$
\norm{\eta (t)}_{\sigma\left(\frac{T_0}{2}\right)}^{2} \leq 1+2 C_{\ep,\beta,\delta,\eta_0}^{2} \;\;\text{for}\; t \in [T_0,\frac{3T_0}{2}].
$$
We mimic this argument by adding $\frac{T_0}{2}$ each time, so we  obtain that $\eta(t)$ is analytic in the time invervals $[\frac{3T_0}{2},2T_0], [2T_0, \frac{5T_0}{2}]...$ Recalling local analyticity \eqref{localanalytic}, we finally obtain 
\begin{equation*}
\norm{\eta (t)}_{\sigma(t)}^{2} \leq  1+2 C_{\ep,\beta,\delta,\eta_0}^{2} \;\text{for any}\; t > 0.
\end{equation*}
This completes our proof.
\end{proof}

\begin{remark}
It is of interest to point out that the global analyticity of the solution which we show here is better than the result in \cite{granero2015nonlocal}, where the argument given by the authors can be extended to prove the global analyticity of solutions of the Kuramoto-Sivashinsky equation, outside a set of time instants with zero measure.
\end{remark}
\section{Existence of attractor and the number of peaks}
In this section, we are interested in the existence of the attractor and its properties. By applying the Theorem 1.1 in \cite{temam1997infinite}, we can prove that the initial value problem \eqref{KStwo-dimension} possesses a compact global attractor in $H^2(\TT^2)$.
First, we denote by $S(t)$ the solution operator, where $S(t)\eta_{0}= \eta(x,y,t)$. 
\begin{defi}\label{compactflow}
The solution operator $S(t)\eta_{0} = \eta(x,y,t)$ defines a compact semiflow in $H^{2}$, if for every initial data $\eta_{0} \in H^{2}$, the following four statements holds:\\
(i) $S(0)\eta_{0}= \eta_{0}$;\\
(ii) $S(t+s)\eta_{0}=S(t)S(s)\eta_{0}, \;\text{for all}\; t,s$;\\
(iii) For every $t>0$, 
$$S(t)(\cdot) : H^{2} \to H^{2}$$
is continuous;\\
(iv) There exists $T^{*}>0$ such that $S(T^{*})$ is a compact operator, i.e. for every bounded set $B \subset H^{2}$, $S(T^{*})B \subset H^{2}$ is a compact set. 
\end{defi}
\begin{defi}
An attractor $\mathcal{A} \subset H^2 $ is a set that satisfies the following properties:\\
(i) $\mathcal{A}$ is an invariant set, i.e., $ S(t)\mathcal{A} = \mathcal{A}, \forall t \geq 0$;\\
(ii) there is an open neighborhood $\mathcal{B} \subset H^2$ of $\mathcal{A}$ such that, for every $\eta_{0}$ in $\mathcal{B}$, $S(t) \eta_{0}$ converges to $\mathcal{A}$ as $t \to \infty$ :
$$
\text{dist}\; (S(t) \eta_{0}, \mathcal{A}) \to 0 \;\; \text{as} \;\;t \to \infty.
$$
\end{defi}

We shall need the following lemma:
\begin{lem}\label{compactattactor}
Let $\eta_{0}\in H^{2}_{\text{od}}(\TT^{2})$ be the initial data. Then $S(\cdot)\eta_{0} \in \mathcal{C} ([0,T];H^{2}_{\text{od}}(\TT^2))$ defines a compact semiflow in $H^{2}_{\text{od}}(\TT^{2})$.   
\end{lem}
\begin{proof}[proof]

In order to show that $S(t)\eta_{0}= \eta(\cdot,\cdot,t)$ defines a compact semiflow, we must verify (i)-(iv) in definition \ref{compactflow}. If we fix a $t_{0}$, the continuity of $S(t_{0})(\cdot)$ from $H^{2}_{\text{od}}$ to $H^{2}_{\text{od}}$ is strightforward by energy estimates. Then, as in Theorem \ref{absorbingsetH4}, we have the existence of an absorbing set in $H^{2}_{\text{od}}$-norm, so there exists $T^{*}$ such that 
$$\norm{\eta(t)}_{H^{2}_{\text{od}}} \leq R'_{\ep, \delta, \beta, \eta_{0}}, \;\forall t \geq T^{*}\;.$$ 
Since (i) and (ii) are obvious, we conclude our proof by invoking the analyticity of solutions.  
\end{proof}
\begin{teo}\label{attractor}
The system \eqref{KStwo-dimension} has a maximal, connected, compact attractor in the space $H^{2}_{\text{od}}(\TT^{2})$.
\end{teo}
\begin{proof}[proof]
By applying Theorem 1.1 in \cite{temam1997infinite} and Lemma \ref{compactattactor}.
\end{proof}

The rest of this section is devoted to studying a particular feature of the chaotic behavior of \eqref{KStwo-dimension}, namely, the number of spatial oscillations. We shall need the lemma proved by Gruji{\'c} in \cite{grujic2000spatial}, which gives us an effective method to study the number of peaks (see also \cite{granero2015nonlocal, burczak2016generalized}). We cannot use directly the method in \cite{grujic2000spatial}, mainly because Lemma 8.1 in \cite{grujic2000spatial} is quite suitable to bound the number of peaks in one space dimension, but not appropriate for our two-dimensional model.  

We first let $\eta(x,y)=f_{y}(x)$, where $f_y(x)$ depends on both $x$ and $y$. Then our original problem \eqref{KStwo-dimension} deduces to 
\begin{equation}\label{KSF}
\pat f_y + f_y \pax f_y + (\beta - 1 ) \pax^2 f_y + F_y (x) = 0, 
\end{equation}
where $F_y(x) = -\eta_{yy} (x,y)- \delta \Lambda^{3}(\eta)  (x,y)+ \ep \Delta^{2} \eta  (x,y)$ can be seen as a forcing term.

In the previous section, we have shown that $\eta(x,y,t)$ is analytic in a growing complex strip
$$
\mathbb{S}_{\sigma(t)}=\{(x,y)+i(\tilde{x}, \tilde{y}):(x,y) \in \TT^{2}, |(\tilde{x}, \tilde{y})| < \sigma(t) \}.
$$ 
Recalling that $\eta(x,y,t)=f_{y}(x,t)$, then $f_{y}(x,t)$ is analytic in 
$$
\mathbb{S}'_{\sigma(t)}=\{x+i\tilde{x}: x \in \TT, |\tilde{x}| < \sigma(t) \}.
$$ 

Now we use the following Lemma (from \cite{grujic2000spatial}):
\begin{lem}\label{numberlemma}
Fix $y$ in problem \eqref{KSF}. Let $\sigma > 0$, and let $f_{y}(x)$ be analytic in the neighbourhood of $\{ z=x+i\tilde{x}: | \Im z| \leq \sigma \}$ and be $2
\pi$-periodic in $x$-direction. Then, for any $\mu >0$, $\TT= I_{\mu} \cup R_{\mu}$, where $I_{\mu}$ is a union of at most $[\frac{4\pi}{\sigma}]$ intervals open in $\TT$, and 
\begin{itemize}
\item $|\pax f_{y}(x) | \leq \mu$, for all $x \in I_{\mu}$;
\item $\text{card } \{ x \in R_{\mu}: \pax f_{y}(x) = 0 \} \leq \frac{2}{\log 2} \frac{2\pi}{\sigma} \log \left(\frac{\max_{|\Im z| \leq \sigma} |\pax f_{y}(z)| }{\mu}\right)$.
\end{itemize}  
\end{lem}

With the help of the lemma above, we have our main result. 
\begin{teo}\label{numberteo}
Let $\eta$ be a solution of system (\ref{KStwo-dimension}) for initial data $\eta_{0} \in H^{2}_{\text{od}}(\TT^{2})$ and define $T_0$ as in Theorem \ref{analyticteo}. Then, $\TT= I \cup R$, where $I$ is a union of at most $[\frac{4\pi}{\tanh(\frac{T_0}{2})}]$ open intervals in $\TT$ and the following estimates hold for $t\geq\frac{T_0}{2}$,
$$
|\pax \eta(x,y,t) | \leq 1, \text{ for all }x \in I, y\in \TT
$$
and 
$$ \cardd \{ x \in R: |\nabla \eta(x,y,t)| = 0\} \leq \frac{4\pi}{\log 2} \frac{\log {C}_{\ep,\beta, \delta,\eta_0}}{\tanh\left(\frac{T_0}{2}\right)}.$$
where ${C}_{\ep,\beta, \delta,\eta_0}$ is a constant depending on $\ep,\beta, \delta, \eta_{0}$.
\end{teo}
\begin{proof}[proof]
From the results of Theorem \ref{analyticteo} and Theorem \ref{attractor}, we know that the system has an attractor and the solution $\eta$ is analytic at least in the strip of width $ \sigma(t)$. 

Now, we can apply Lemma \ref{numberlemma} with $\mu = 1 $ and bound
\begin{align*}
\cardd \{ x \in R: |\nabla \eta(x,y,t)|=0\} 
&\leq \cardd \{ x \in R: \pax \eta = 0\} \\
&\leq \frac{2}{\log 2} \frac{2\pi}{\sigma} \log \left(\frac{\max_{|\Im z| \leq \sigma} |\pax \eta(z,y,t)| }{\mu}\right)\\
&\leq \frac{4\pi}{\log 2} \frac{1}{\sigma} \log \left(C \norm{e^{\sigma(t) \Lambda} \eta }_{L^{2}}\right).
\end{align*}
In the last inequality above, we used the following estimate
\begin{align*}
\normm{\pax \eta (z,y,t)}_{L^{\infty}(|\Im z|\leq \frac{\sigma}{2})}
& =\normm{\pax \eta (x+i\tilde{x},y,t)}_{L^{\infty}(|\Im z|\leq \frac{\sigma}{2})} \\
& \leq \normm{\sum_{(\xi_{1},\xi_{2}) \in \ZZ^{2}}  |\xi_{1}|\widehat{\eta} (\xi_{1},\xi_{2},t) e^{i(x\xi_{1}+y\xi_{2})} e^{-\tilde{x}\xi_{1}} }_{L^{\infty}(|\Im z|\leq \frac{\sigma}{2})} \\
& \leq \normm{\sum_{(\xi_{1},\xi_{2}) \in \ZZ^{2}}  | \xi| |\widehat{\eta} (\xi_{1},\xi_{2},t) |  e^{|\tilde{x}||\xi|}}_{L^{\infty}(|\Im z|\leq \frac{\sigma}{2})} \\
& \leq \sum_{(\xi_{1},\xi_{2}) \in \ZZ^{2}} |\xi| |\widehat{\eta} (\xi_{1},\xi_{2},t)|  e^{\frac{\sigma(t)}{2}|\xi|}\\
& \leq \sum_{(\xi_{1},\xi_{2}) \in \ZZ^{2}} |\xi|^3|\widehat{\eta} (\xi_{1},\xi_{2},t)|  e^{\frac{\sigma(t)}{2}|\xi|} \frac{1}{|\xi|^2}\\
& \leq \sum_{(\xi_{1},\xi_{2}) \in \ZZ^{2}} \widehat{\eta} (\xi_{1},\xi_{2},t)|  e^{\sigma(t)|\xi|}  \frac{1}{|\xi|^2} \\
& \leq \norm{e^{\sigma(t) \Lambda} \eta }_{L^{2}}  \left(\sum_{(\xi_{1},\xi_{2}) \in \ZZ^{2}} \frac{1}{|\xi|^4} \right)^{\frac{1}{2}}\\
& \leq C \norm{e^{\sigma(t) \Lambda} \eta }_{L^{2}} 
\end{align*} 
where we used the fact that 
$$|\xi|^3  e^{\frac{\sigma(t)}{2}|\xi|} \leq C  e^{\sigma(t)|\xi|}.$$
Since $\eta$ has global analyticity 
$$\norm{e^{\sigma(t) \Lambda} \eta (t)}_{L^{2}}\leq  1+2 {C}_{\ep,\beta, \delta,\eta_0}, \;\forall\;t \geq 0,$$
we can conclude that for $t \geq \frac{T_0}{2}$
\begin{align*}
\cardd \{ x \in R: |\nabla \eta(x,y,t)|= 0\} 
\leq \frac{4\pi}{\log 2} \frac{\log {C}_{\ep,\beta, \delta,\eta_0}}{\tanh\left(\frac{T_0}{2}\right)}.
\end{align*}
where $C_{R_{\ep,\beta,\delta,\eta_0}}$ depends on $R_{\ep,\delta,\beta}$ and $\eta_0$.
\end{proof}
Theorem \ref{numberteo} gives us a bound of the number of wild spatial oscillations of the solution, then the following corollary is a direct result of it.
\begin{corol}
Let $\eta$ be a solution corresponding to the initial data $\eta_{0} \in H^{2}_{\text{od}}(\TT^{2})$, then for $t \geq \frac{T_0}{2}$, the number of peaks for $\eta$ can be bounded as
$$ \cardd \; \{\text{peaks for} \;\eta \} \leq \frac{4\pi}{\log 2} \frac{\log {C}_{\ep, \beta,\delta,\eta_0}}{\tanh\left(\frac{T_0}{2}\right)}.$$
where $C_{\ep,\beta, \delta,\eta_0}$ depends on $\ep,\beta,\delta,\eta_{0}$ and $T_0$ is defined as before.
\end{corol} 

\section{Numerical simulations}
In this section, we show some numerical solutions of the initial-value problem \eqref{KStwo-dimension} with the following initial condition 
\begin{equation}\label{initialdata}
\eta_{0}(x,y) = - \sin x \left(\sin y+ e^{-y^2} \cos y\right).
\end{equation}

As the equation is periodic, we discretize the spatial part by means of a Fourier spectral method. Namely, taking the Fourier transformation of equation \eqref{KStwo-dimension}, we get 
\begin{equation*}
    \widehat{\eta}_t + \frac{i \xi_1}{2} \widehat{\eta^2}  + \left(- (\beta-1) \xi_1^2 + \xi_2^2 - \delta |\xi|^3 + \ep |\xi|^4\right) \widehat{\eta} = 0
\end{equation*}
where $|\xi| = \sqrt{\xi_1^2 + \xi_2^2}$. Once we discretize the spatial part of the PDE we get a system of ODE : 
\begin{equation}\label{ODE}
\widehat{\eta}_t = 	\mathbf{L}\, \widehat{\eta} + \mathbf{N}\, \widehat{\eta}
\end{equation}
with
\begin{equation*}
	(\mathbf{L}\, \widehat{\eta}) (\xi_1, \xi_2) =   \left( (\beta-1) \xi_1^2 - \xi_2^2 + \delta |\xi|^3 - \ep |\xi|^4\right) \widehat{\eta}
\end{equation*}
and 
\begin{equation*}
	(\mathbf{N}\, \widehat{\eta}) (\xi_1, \xi_2) = - \frac{i \xi_1}{2} \mathcal{F} \left(\left( \mathcal{F}^{-1}(\widehat{\eta} ) \right)^2 \right)
\end{equation*}
We then compute the numerical solution by using a fourth-order exponential time differencing (ETD-RK4) method that was first derived by Cox and Mathews in \cite{cox2002exponential} and then was improved by Kassam and Trefethen in \cite{kassam2005fourth}. 

In the following figures, we show a numerical solution of \eqref{ODE} with parameters $\beta=2, \delta=0.5, \epsilon=1$ and initial condition \eqref{initialdata}. We can see that the equation \eqref{KStwo-dimension} is very interesting from a dynamical systems point of view, as it is a PDE that can exhibit
chaotic solutions.
\begin{figure}
		\centering
		\includegraphics[width=12cm]{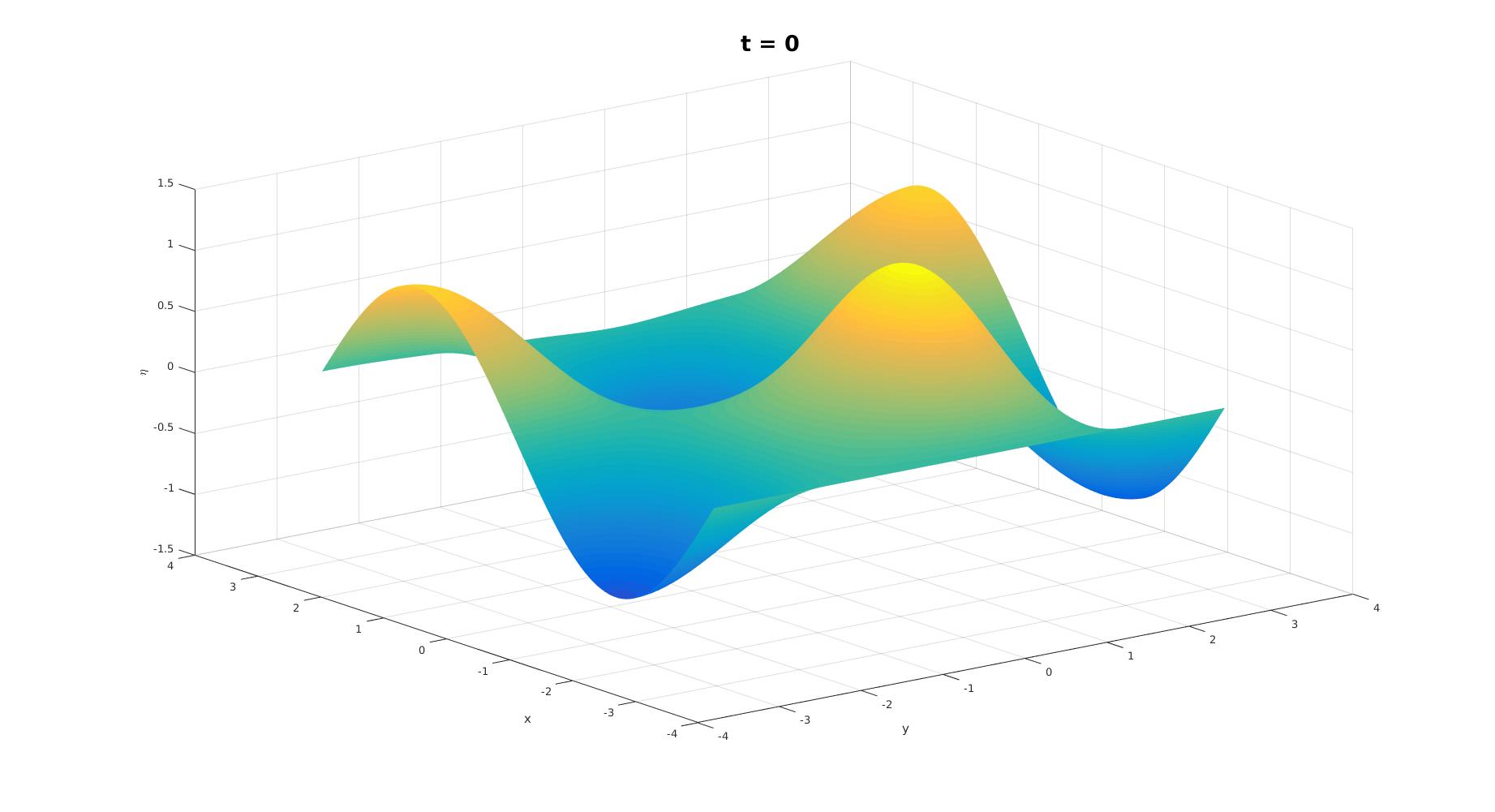}
	   	\caption{Initial data \eqref{initialdata}}
		\label{initial}	
\end{figure}

\begin{figure}[h]
	\centering
	\includegraphics[width=12cm]{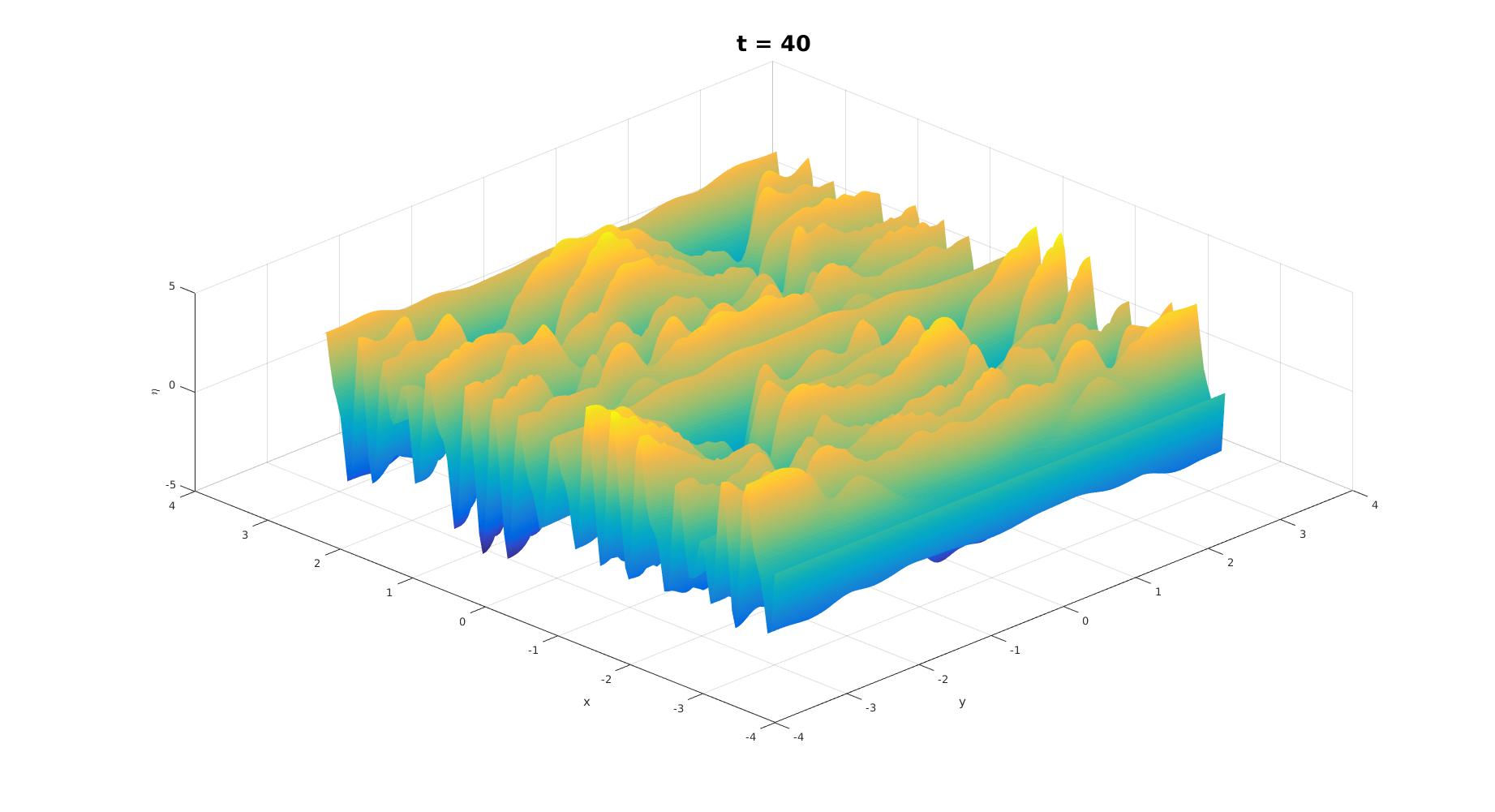}	
	\caption{Numerical solution of \eqref{ODE} for $\beta=2, \delta=0.5, \epsilon=1$ at $t=40$.}
    \label{end}	
\end{figure}

\begin{figure}
	\centering
	\includegraphics[width=15cm]{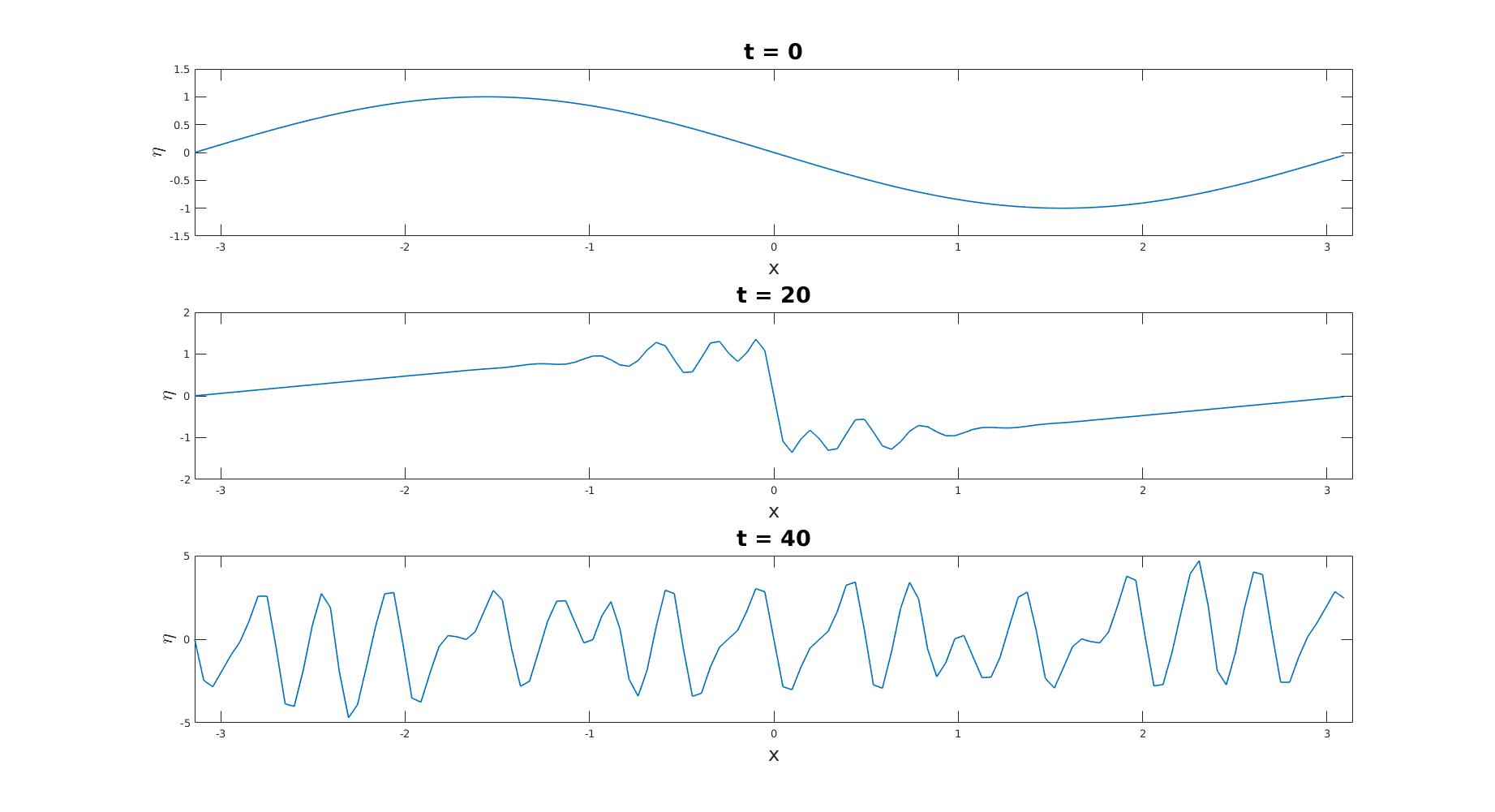}
	\caption{Numerical solution profile along $x$-direction for $y=0$ of \eqref{ODE} with $\beta=2, \delta=0.5, \epsilon=1$ and the same initial data as in Figure \ref{end} at $t=0, 20, 40$.}
	\label{cor_x}	
\end{figure}

\begin{figure}
	\centering
	\includegraphics[width=15cm]{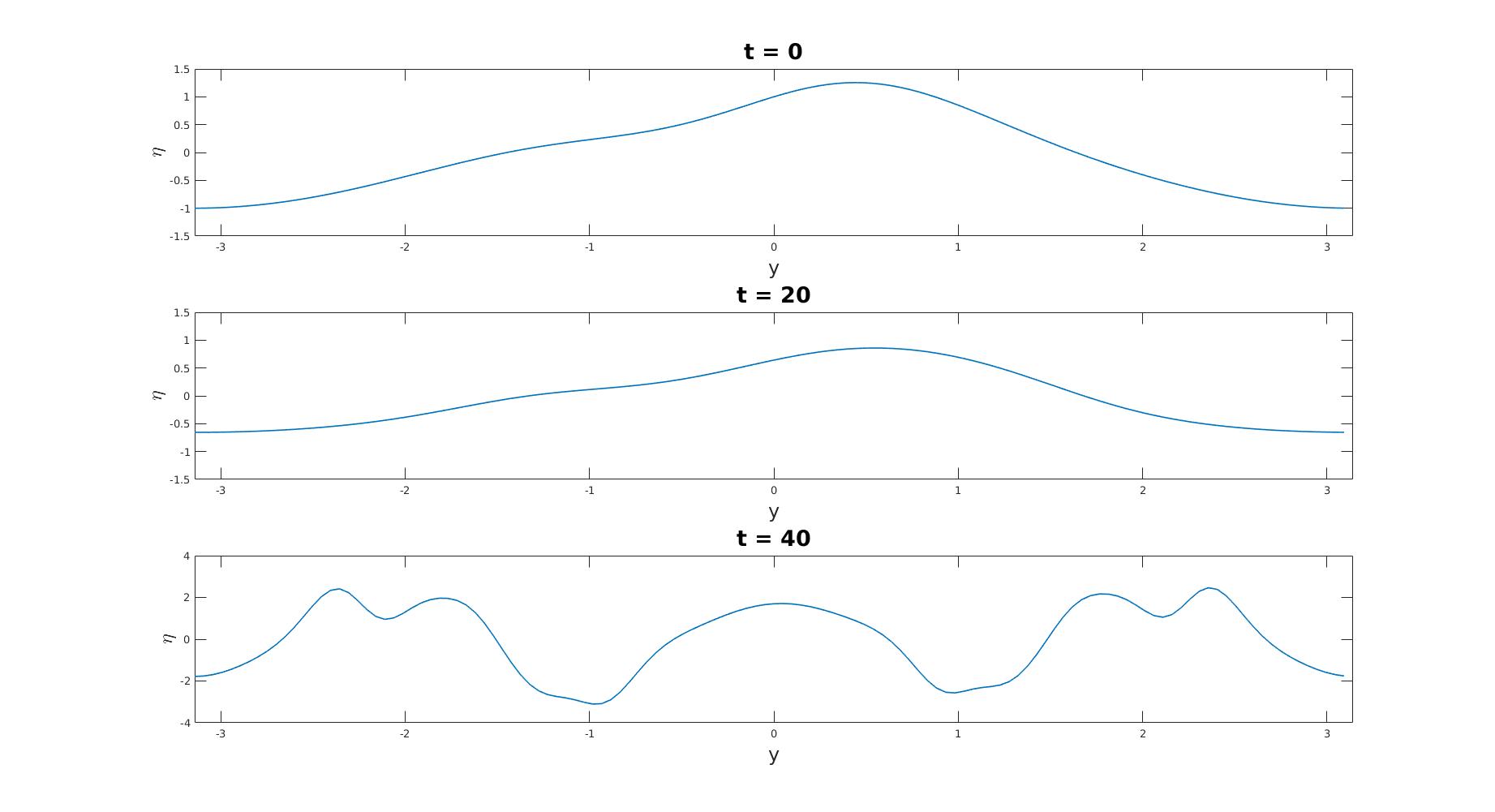}
	\caption{Numerical solution profile along $y$-direction for $x=\pi/2 $ of \eqref{ODE} with $\beta=2, \delta=0.5, \epsilon=1$ and the same initial data as in Figure \ref{end} at $t=0, 20, 40$.}
	\label{cor_y}	
\end{figure}

\section*{Acknowledgment}
The authors would like to express sincere gratitude to Ruben Tomlin for calling our attention to an error in our first manuscript, and thank Lorenzo Brandolese and 
Drago\c s Iftimie for providing helpful suggestions and comments. 
R.GB was supported by the LABEX MILYON (ANR-10-LABX-0070) of Universit\'e de Lyon, within the program ``Investissements d'Avenir" (ANR-11-IDEX-0007) operated by the French National Research Agency (ANR). J.H has been partially funded by the ANR project Dyficolti ANR-13-BS01-0003-01. 

\newpage

\bibliographystyle{abbrv}
%\bibliography{KSbibtex}

\bigskip

\begin{description}
	\item[J. He] Universit\'e de Lyon, Universit\'e Lyon 1 --
	CNRS UMR 5208 Institut Camille Jordan --
	43 bd. du 11 Novembre 1918 --
	Villeurbanne Cedex F-69622, France.\\
	Email: \texttt{jiao.he@math.univ-lyon1.fr}
	\item[R. Granero-Belinch\'{o}n] Departamento  de  Matem\'aticas,  Estad\'istica  y  Computaci\'on,  Universidad  de Cantabria.  Avda.  Los  Castros  s/n,  Santander,  Spain.\\
	Email: \texttt{rafael.granero@unican.es}
\end{description}

\end{document}